\documentclass[bj]{imsart}
\usepackage{oxford3}
\usepackage{kokoszka}
\usepackage{graphicx}
\usepackage{amsfonts}
\usepackage{rotating}
\usepackage{verbatim}
\usepackage{amsmath}
\usepackage[compact]{titlesec}
\usepackage{color}
\usepackage{mathrsfs}
\usepackage{bbm}

\newcommand\cL{{\mathcal L}}

\newcommand\cS{{\mathcal S}}
\newcommand\cT{{\mathcal T}}

\newcommand\jp{j^\prime}

\newcommand\ip{i^\prime}

\renewcommand\QED{\hfill\ \hbox{\hskip6pt\vrule height 8pt width 6pt}}

\begin{document}

\begin{frontmatter}

\title{Principal components analysis of regularly varying functions}
\author{Piotr Kokoszka\\
{\small Colorado State University}
 \and
Stilian Stoev\\
{\small University of Michigan}
\and
Qian Xiong\\
{\small Colorado State University}
}

\date{}
\maketitle

\begin{abstract}
\noindent The paper is concerned with asymptotic properties of the
principal components analysis of functional data. The currently
available results assume the existence of the fourth moment.
We develop analogous results in a setting which does not
require this assumption. Instead, we assume that the observed
functions are regularly varying. We derive the asymptotic distribution
of the sample covariance operator and of the sample functional principal
components. We obtain a number of results on the convergence of
moments and almost sure convergence. We apply the new theory
to establish the consistency of the regression operator in a functional
linear model.

\medskip


\end{abstract}

\end{frontmatter}


\section{Introduction}\label{s:i}
A fundamental technique of functional data analysis is to replace
infinite dimensional curves by coefficients of their projections onto
suitable, fixed or data--driven, systems, e.g. \citetext{bosq:2000},
\citetext{ramsay:silverman:2005}, \citetext{HKbook},
\citetext{hsing:eubank:2015}. A finite number of these coefficients
encode the shape of the curves and are amenable to various statistical
procedures. The best systems are those that lead to low dimensional
representations,  and so provide the most efficient dimension reduction.
Of these, the functional principal components (FPCs) have been most
extensively used, with hundreds of papers dedicated to the various
aspects of their theory and applications.

If $X, X_1, X_2, \ldots, X_N$ are mean zero
iid functions in $L^2$ with $E\lnorm X \rnorm^2 < \infty$,
then
\begin{equation} \label{e:KL}
X_n(t) = \sum_{j=1}^\infty \xi_{nj} v_j(t), \ \ \ E\xi_{nj}^2 = \la_j.
\end{equation}
The FPCs $v_j$ and the eigenvalues $\la_j$ are,  respectively,
the eigenfunctions and the eigenvalues of the covariance operator
$C: L^2 \to L^2$ defined by $C(x)(t) = \int \cov(X(t), X(s)) x(s) ds$
As such, the $v_j$   are orthogonal. We assume they are
normalized to unit norm. The $v_j$ form  an optimal orthonormal
basis for dimension reduction measured by the $L^2$ norm, see e.g.
Theorem 11.4.1 in \citetext{KRbook}.

The $v_j$ and the $\la_j$ are estimated by $\hat v_j$ and
$\hat\la_j$ defined by
\begin{equation} \label{e:emp}
\int \hat c(t,s) \hat v_j(s) ds = \hat \la_j \hat v_j(t),
\end{equation}
where
\begin{equation} \label{e:hat-c}
\hat c(t,s) = \frac{1}{N} \sum_{n=1}^N X_n(t) X_n (s).
\end{equation}
Like the $v_j$, the $\hat v_j$ are defined only up to a sign. 
Thus, strictly speaking, in the formulas that follow, the $\hat v_j$  
would need to be 
replaced with $\hat c_j \hat v_j$, where
$\hat c_j = {\rm sign}\lip \hat v_j , v_j\rip$. As is customary, 
to lighten the notation, we assume that the orientations of 
 $v_j$ and $\hat v_j$ match, i.e. $\hat c_j =1$. 

Under the existence of the fourth moment,
\begin{equation} \label{e:4th}
E\lnorm X \rnorm^4 = \lbr \int X^2(t) dt \rbr^2 < \infty,
\end{equation}
and assuming $\la_1 > \la_2>  \ldots$,
it has been shown that  for each $j \ge 1$,
\begin{equation} \label{e:rate}
\limsup_{N\to \infty} N E\lnorm  \hat v_j - v_j \rnorm^2 < \infty, \ \ \
\limsup_{N\to \infty} N E\lp \hat\la_j - \la_j \rp^2 < \infty,
\end{equation}
\begin{equation} \label{e:clt-laj}
N^{1/2}(\hat\la_j - \la_j) \convd N(0, \sg_j^2),
\end{equation}
\begin{equation} \label{e:clt-vj}
N^{1/2}(\hat v_j - v_j) \convd N(0, C_j),
\end{equation}
for a suitably defined variance $\sg_j^2$ and a
covariance operator $C_j$. The above relations, especially \refeq{rate},
have been used to derive large sample justifications of inferential
procedures based on the estimated FPCs $\hat v_j$. In most scenarios,
one can show that replacing the $\hat v_j$ by the $v_j$ and
the $\hat\la_j$ by the $\la_j$ is asymptotically
negligible. Relations \refeq{rate} were established by
\citetext{dauxois:1982} and extended to weakly dependent
functional time series by \citetext{hormann:kokoszka:2010}.
Relations \refeq{clt-laj} and \refeq{clt-vj} follow from
the results of \citetext{kokoszka:reimherr:2013}. In case
of continuous functions satisfying regularity conditions, they
follow from the results of \citetext{hall:h-n:2006}. 

A crucial assumption for the relations \refeq{rate}--\refeq{clt-vj} to
hold is the existence of the fourth moment, i.e. \refeq{4th}, the iid
assumption can be relaxed in many ways. Nothing is at present known
about the asymptotic properties of the FPCs and their eigenvalues if
\refeq{4th} does not hold.  Our objective is to explore what can be
said about the asymptotic behavior of $\widehat C$, $\hat v_j$ and
$\hat\la_j$ if
\refeq{4th} fails. We would thus like to consider the case of
$E\| X_n \|^2 < \infty$ and $E\| X_n \|^4 = \infty$.  Such an assumption
is however too general. From mid 1980s to mid 1990s similar
questions were posed for scalar time series for which the fourth
or even second moment does not exist. A number of results
pertaining to the convergence of sample covariances and
the periodogram  have been derived under the assumption
of regularly varying tails, e.g.
Davis and Resnick
(\citeyear{davis:resnick:1985}, \citeyear{davis:resnick:1986}),
\citetext{kluppelberg:mikosch:1994}, \citetext{mikosch:gka:1995},
\citetext{kokoszka:taqqu:1996}, \citetext{anderson:meerschaert:1997};
many others are summarized in the monograph of
\citetext{embrechts:kluppelberg:mikosch:1997}. The assumption of
regular variation is natural because non--normal stable limits
can be derived   by establishing
a connection to random variables in a stable domain of attraction, which
is characterized by regular variation. This is the approach we take.
We assume that the functions $X_n$ are regularly varying in the
space $L^2$ with the index $\ag\in (2, 4)$, which implies
$E\| X_n \|^2 < \infty$ and $E\| X_n \|^4 = \infty$. Suitable
definitions and assumptions are presented in Section~\ref{s:prelim}.

The paper is organized as follows.  The remainder of the introduction
provides a practical motivation for the theory we develop.
It is not necessary to understand the contribution of the
paper, but, we think, it gives  a good feel for what is being studied. The
formal exposition begins in Section~\ref{s:prelim}, in which notation
and assumptions are specified.  Section~\ref{s:hatC} is dedicated to
the convergence of the sample covariance operator (the integral
operator with kernel \refeq{hat-c}).  These results are then used in
Section~\ref{s:KR} to derive various convergence results for the
sample FPCs and their eigenvalues. Section~\ref{s:app} shows
how the results derived in previous sections can be used in a context
of a  functional regression model. Its objective is to illustrate the
applicability of our theory in a well--known and extensively studied
setting. It is hoped that it will motivate and guide applications to other
problems of functional data analysis. All proofs which go beyond simple
arguments are presented in Online material.

\begin{figure}
\begin{centering}
\includegraphics[width=0.95\textwidth]{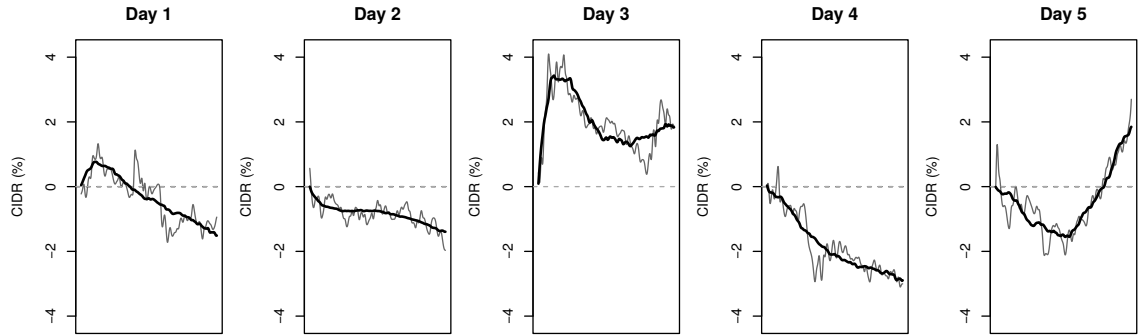}
\par\end{centering}
\caption{Five consecutive intraday return curves,  Walmart stock.
The raw returns are noisy grey lines. The smoother black
lines are approximations
$\widehat X_i(t) = \sum_{j=1}^3 \hat\xi_{ij} \hat v_j$.
\label{f:5day2}}
\end{figure}

We conclude this introduction by presenting a specific data context.
Denote by $P_i(t)$ the price of an asset at time $t$ of trading day
$i$. For the assets we consider in our illustration, $t$ is time in minutes
between 9:30  and and 16:00 EST (NYSE opening times) rescaled
to the unit interval $(0,1)$. The intraday return curve on day $i$ is
defined by
$X_i(t) = \log P_i(t) - \log P_i(0)$. In practice, $P_i(0)$ is the
price after the first minute of trading. The curves $X_i$ show how
the return accumulates over the trading day, see e.g.
\citetext{lucca:moench:2015}; examples of
 are shown in Figure~\ref{f:5day2}.

\begin{figure}
\begin{centering}
\includegraphics[width=0.75\textwidth]{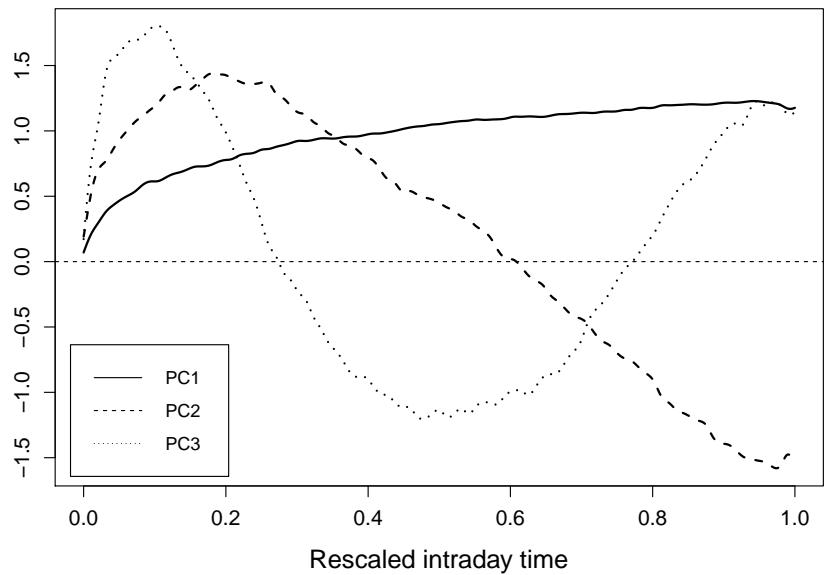}
\par\end{centering}
\caption{The first three sample FPCs of intraday returns on
Walmart stock. \label{f:fpc}}
\end{figure}

The first three sample FPCs, $\hat v_1, \hat v_2, \hat v_3$, are shown
in Figure~\ref{f:fpc}. They  are computed, using
\refeq{emp},  from minute-by-minute
Walmart returns form July 05, 2006 to Dec 30, 2011, $N=1,378$ trading
days. (This time interval is used for the  other assets we consider.)
The curves $\widehat X_i = \sum_{j=1}^3 \hat\xi_{ij} \hat v_j$,
with the scores $ \hat\xi_{ij} = \int X_i(t) \hat v_j(t)dt$,  visually
approximate the curves $X_i$ well.
One can thus expect that
the $\hat v_j$ (with properly adjusted sign) are good estimators  of
the population FPCs $v_j$  in \refeq{KL}. Relations \refeq{rate} and
\refeq{clt-vj} show that this is indeed the case,  {\em if}
$E \| X\|^4 < \infty$.
(The curves $X_i$ can be assumed to form a stationary time series in $L^2$,
see \citetext{horvath:kokoszka:rice:2014}.)
We will now argue that the assumption of the finite fourth moment
is not realistic, so, with the currently available theory, it is not clear
if the $\hat v_j$ are good estimators of the $v_j$.  If $E \| X\|^4 < \infty$,
then $E \xi_{1j}^4 < \infty$ for every $j$. Figure \ref{f:wmt-ibm}
shows the Hill plots of the
sample score $\hat\xi_{ij}$ for two stocks and for $j=1, 2, 3$.
Hill plots for other blue chip stocks look similar. 
These plots illustrate several properties. 1) It is reasonable
to assume that the scores have Pareto tails.
2) The tail index $\ag$ is smaller than 4, implying that the fourth
moment does not exist. 3)
It is reasonable to assume that the tail index does not depend on $j$
and is between 2 and 4. 
With such a motivation, we are now able to formalize in the next section
the setting of this paper.

\begin{figure}
\begin{centering}
\includegraphics[width=1.05\textwidth]{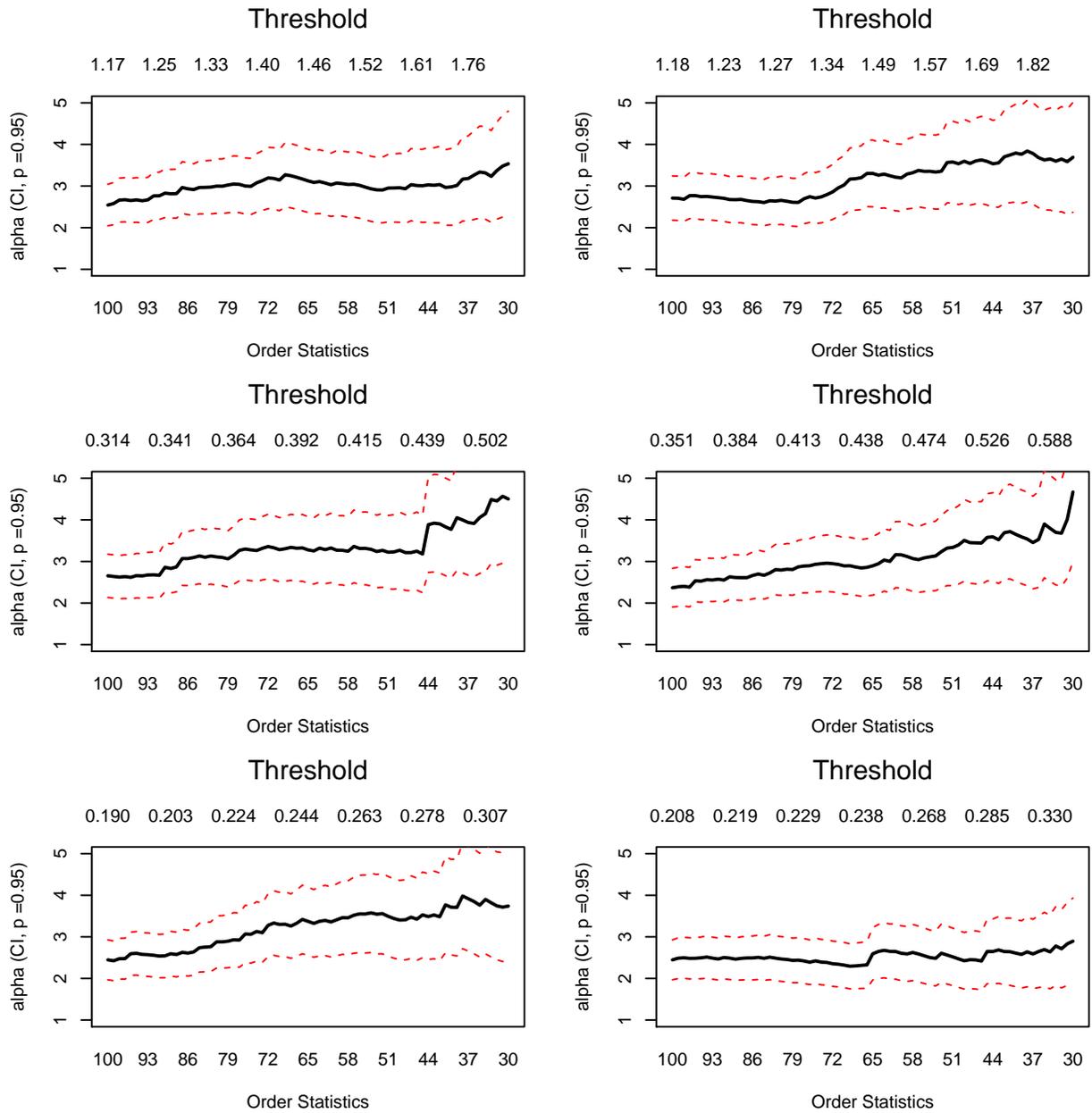}
\par\end{centering}
\caption{Hill plots (an estimate of $\ag$ as a function
of upper order statistics)   for sample FPC scores for {\bf Walmart} (left)
and {\bf IBM} (right). From top to bottom:  levels $j= 1, 2, 3$.
\label{f:wmt-ibm}}
\end{figure}


\section{Preliminaries} \label{s:prelim} The functions $X_n$ are
assumed to be independent and identically distributed in $L^2$, with
the same distribution as $X$, which is regularly varying with index
$\ag\in (2,4)$. By $L^2 := L^2(\cT)$, we denote the usual
separable Hilbert space of square integrable functions on some compact
subset $\cT$ of an Euclidean space. In a typical FDA framework, $\cT =
[0,1]$, e.g. Chapter 2 of \citetext{HKbook}.  Regular variation in
finite--dimensional  spaces has been a topic of extensive research for decades,
see e.g.\ Resnick (\citeyear{resnick:1987,resnick:2006}) and
\citetext{meerschaert:scheffler:2001book}. We shall need the concept of
regular variation of measures on {\em infinitely-dimensional} function
spaces. To this end, we start by recalling some terminology and
fundamental facts about regularly varying functions.

A measurable function $L:(0,\infty) \to \mathbb R$ is
said to be slowly varying (at infinity) if, for all $\lambda>0$,
\[
\frac{L(\lambda u)}{L(u)} \to 1,\ \quad \mbox{ as }u\to\infty.
\]
Functions of the form $R(u) = u^\rho L(u)$
are said to be regularly varying with
exponent $\rho\in{\mathbb R}$.

The notion of regular variation extends to measures and provides an
elegant and powerful framework for establishing limit theorems. It was
first introduced by \citetext{meerschaert:1984} and has been since
extended to Banach and even metric spaces using the notion of $M_0$
convergence (see e.g.\ \citetext{hult:lindskog:2006}).  Even though we
will work only with Hilbert spaces, we review the theory in a more
general context.

Consider a separable Banach space $\mathbb{B}$ and let $B_\epsilon:=
\{z\in\mathbb{B}\, :\, \|z\|<\epsilon\}$ be the open ball of radius
$\epsilon>0$, centered at the origin. A Borel measure $\mu$ defined on
$\mathbb{B}_0:=\mathbb{B} \backslash \{ \mathbf{0}\}$ is said to be
{\em boundedly finite} if $\mu(A)<\infty$, for all Borel sets that are
bounded away from $\mathbf{0}$, that is, such that $A\cap B_\epsilon =
\emptyset$, for some $\epsilon>0$.  Let $\mathbb{M}_0$ be the
collection of all such measures.  For $\mu_n,\mu\in\mathbb{M}_0$, we
say that the $\mu_n$ converge to $\mu$ in the $M_0$ topology, if
$\mu_n(A) \to \mu(A)$, for all bounded away from $\mathbf{0}$,
$\mu$-continuity Borel sets $A$, i.e., such that $\mu(\partial A) =0$,
where $\partial A:= \overline {A} \setminus A^{\circ}$ denotes the
boundary of $A$. The $M_0$ convergence can be metrized such that
$\mathbb{M}_0$ becomes a complete separable metric space (Theorem 2.3
in \citetext{hult:lindskog:2006} and also Section 2.2. of
\citetext{meinguet:2010}). The following result is known, see e.g.
Chapter 2 of \citetext{meinguet:2010} and references therein.

\begin{proposition}\label{p:RV}
Let $X$ be a random element in a separable
Banach space $\mathbb{B}$ and $\ag >0$.
The following three statements are equivalent:
\begin{itemize}
\item[(i)] For some slowly varying function $L$,
\begin{equation} \label{e:Xtail}
P(\lnorm X \rnorm > u ) = u^{-\ag} L(u)
\end{equation}
and
\begin{equation} \label{e:X-RV-mu}
\frac{P(u^{-1}X \in  \cdot )}{P(\lnorm X \rnorm>u)}
\stackrel{M_0}{\longrightarrow} \mu(\cdot), \ \ u \to \infty,
\end{equation}
where $\mu$ is a non-null measure on the Borel $\sigma$-field
$\mathcal{B}(\mathbb{B}_0)$ of $\mathbb{B}_0=\mathbb{B} \backslash
\lbr \mathbf{0} \rbr$.
\item[(ii)] There exists a probability measure $\Gamma$ on the unit
  sphere $\mathbb{S}$ in $\mathbb{B}$ such that, for every $t>0$,
\[
\frac{P(\lnorm X \rnorm>tu, X/ \lnorm X \rnorm  \in  \cdot )}{P(\lnorm X \rnorm>u)} \stackrel{w}{\longrightarrow}
t^{-\alpha} \Gamma(\cdot), \ \ u \to \infty.
\]
\item[(iii)] Relation \refeq{Xtail} holds, and for the same spectral measure $\Gamma$ in (ii),
\[
P \lp X/ \lnorm X \rnorm \in \cdot | \lnorm X \rnorm >u  \rp
\stackrel{w}{\longrightarrow} \Gamma(\cdot), \ \ u \to \infty.
\]
\end{itemize}
\end{proposition}

\begin{definition}\label{d:RV} If any
 one of the equivalent conditions in Proposition \ref{p:RV} hold, we
 shall say that $X$ is regularly varying with index $\alpha$.  The
 measures $\mu$ and $\Gamma$ will be referred to as exponent and
 angular measures of $X$, respectively.
\end{definition}

The measure $\Gamma$ is sometimes called the spectral measure,
but we will use the adjective ``spectral'' in the context of stable
measures which appear in Section~\ref{s:hatC}. It is important to distinguish
the angular  measure of a regularly varying
random function and a spectral measure of a stable
distribution, although they are related. We also note that we
call $\ag$ the tail index, and $-\ag$ the tail exponent.

We will work under the following assumption.
\begin{assumption}\label{a:X-RV}
  The random element $X$ in the separable Hilbert space $H= L^2$ 
  has mean zero and is
  regularly varying with index $\ag \in (2,4)$.
The observations $X_1, X_2, \ldots $ are independent  copies of $X$.
\end{assumption}

Assumption~\ref{a:X-RV} is a coordinate free condition not related
in any way to functional principal components. The next assumption
relates the asymptotic behavior of the FPC scores to the assumed  regular
variation. It implies, in particular, that the expansion
$X(t) = \sum_{j=1}^\infty \xi_j v_j(t)$ contains  infinitely
many terms, so that we study infinite dimensional
objects. We will see in the proofs of Proposition~\ref{p:KM-equiv}
and Theorem~\ref{t:asym-prodX}
that under Assumption~\ref{a:X-RV} the limit
\[
Q_{nm}
= \lim_{u\to \infty}
\frac{P\lp \lbr \sum_{j=n}^\infty \xi_j^2 \rbr^ {1/2}
\lbr \sum_{j=m}^\infty \xi_j^2 \rbr^ {1/2}> u \rp }
{P\lp \sum_{j=1}^\infty \xi_j^2 > u \rp }
\]
exists and is finite. We impose the following
assumption  related to condition \refeq{X-RV-mu}.

\begin{assumption} \label{a:KM} For every $n, m \ge 1$,
$Q_{nm} > 0$.
\end{assumption}
Assumption~\ref{a:KM} postulates, intuitively,
that the tail sums  $\sum_{j=n}^\infty \xi_j^2$ must have
extreme probability tails comparable to that  of $\| X\|^2$.

We now collect several useful facts that will be used in the following.
The exponent measure $\mu$ satisfies
\begin{equation}\label{e:mu-prop}
\mu (tA)=t^{-\alpha} \mu(A), \ \ \forall t>0,
\ \ A \in \mathcal{B}(\mathbb{B}_0).
\end{equation}
It admits the polar coordinate representation via the angular
measure $\Gamma$.  That is, if $x = r\theta$, where $r:= \|x\|$ and
$\theta = x/\|x\|,$ for $x\not=\mathbf{0}$, we have
\begin{equation}\label{e:disintegration}
\mu(dx) = \alpha r^{-\alpha-1} dr \Gamma(d\theta).
\end{equation}
This means that for every bounded measurable function $f$ that
vanishes on a neighborhood of $\mathbf{0}$, we have
\[
 \int_{\mathbb{B}}  f(x) \mu(dx)
= \int_{\mathbb{S}} \int_0^\infty f(r\theta) \alpha r^{-\alpha -1 } dr
\Gamma(d\theta).
\]
There exists a sequence $\lbr a_N\rbr$ such that
\begin{equation} \label{e:covp_x}
NP \lp X \in a_N A \rp \to \mu (A),
\end{equation}
for any set $A$ in $\mathcal{B}(\mathbb{B}_0)$ with $\mu (\partial A)=0$.  One can take, for example,
\begin{equation}\label{e:aN}
a_N=N^{1/\alpha}L_0(N),
\end{equation}
with a slowly varying function $L_0$ satisfying
$L_0^{-\alpha}(N)L(N^{1/\ag}L_0(N)) \to 1$.

We will work with Hilbert--Schmidt operators. A linear operator $\Psi:
H \to H$ is Hilbert--Schmidt if $\sum_{j=1}^\infty \lnorm \Psi(e_j)
\rnorm^2 < \infty$, where $\lbr e_j \rbr$ is any orthonormal basis of
$H$. Every Hilbert--Schmidt operator is bounded.  The space of
Hilbert--Schmidt operators will be denoted by $\cS$. It is itself a
separable Hilbert space with the inner product
\[
\lip \Psi_1, \Psi_2 \rip_\cS
= \sum_{j=1}^\infty \lip \Psi_1(e_j), \Psi_2 (e_j)\rip.
\]
If $\Psi$ is an integral operator defined by
$\Psi(x)(t) = \int \psi(t,s) x(s) ds, \ x \in L^2$,
then $\|\Psi \|_{\cS}^2 = \iint \psi^2(t,s) dt ds$.

Relations \refeq{rate} essentially
follow from the bound
\[
E \lnorm \widehat C - C \rnorm_{\cS}^2 \le N^{-1} E \lnorm X \rnorm^4,
\]
where the subscript $\cS$ indicates the Hilbert--Schmidt norm.  Under
Assumption~\ref{a:X-RV} such a bound is useless because, by
\refeq{Xtail}, $E \lnorm X \rnorm^4 = \infty$.  In fact, one can show
that under Assumption~\ref{a:X-RV}, $E \| \widehat C \|_{\cS}^2 =
\infty$, so no other bound on $E \| \widehat C - C \|_{\cS}^2$ can be
expected.  The following Proposition~\ref{p:HS-as} implies however
that under Assumption \ref{a:X-RV} the population covariance operator
$C$ is a Hilbert-Schmidt operator, and $\widehat C\in \cS$ with
probability 1.  This means that the space $\cS$ does provide a
convenient framework.

\begin{proposition} \label{p:HS-as} Suppose $X$ is a random element
of $L^2$ with $E\| X\|^2< \infty$ and $\widehat C$ is  the sample
covariance operator based on $N$  iid copies of $X$.
Then $C\in \cS$ and $\widehat C\in \cS$ with probability 1.
\end{proposition}

Like all proofs, the proof of Proposition \ref{p:HS-as} is presented
in the on-line material.

\section{Limit distribution of $\widehat C$} \label{s:hatC}
We will  show that $N k_N^{-1} (\widehat C - C)$ converges to an
$\ag/2$--stable Hilbert--Schmidt operator, for an appropriately
defined regularly varying sequence $\lbr k_N \rbr$.  Unless stated
otherwise, all limits in the following are taken as $N
\to \infty$.

Observe that for any $x \in H$,
\begin{align}\label{e:hatC-C}
Nk_N^{-1} \lp \widehat C -C \rp (x)
& =Nk_N^{-1} \lp N^{-1}\sum_{n=1}^N \lip X_n,x \rip X_n-E [\lip X_1,x \rip X_1]  \rp \\
& =k_N^{-1} \lp \sum_{n=1}^N \lip X_n,x \rip X_n-NE [\lip X_1,x \rip X_1]  \rp \notag \\
&=k_N^{-1} \lp \sum_{n=1}^N \lp X_n \otimes X_n \rp (x)
-N E [ \lp X_1 \otimes X_1 \rp](x)  \rp, \notag
\end{align}
where $\lp X_n \otimes X_n \rp (x)=\lip X_n,x \rip X_n$.  Since the
$X_n \otimes X_n$ are Hilbert--Schmidt operators, the last expression
shows a connection between the asymptotic distribution of $\widehat C$
and convergence to a stable limit in the Hilbert space $\cS$ of
Hilbert--Schmidt operators.  We therefore restate below, as
Theorem~\ref{t:KM}, Theorem 4.11 of \citetext{kuelbs:mandrekar:1974}
which provides conditions for the stable domain of attraction in a
separable Hilbert space. The Hilbert space we will consider in the
following will be $\cS$ and the stability index will be $\ag/2, \ag
\in (2, 4)$.  However, when stating the result of Kuelbs and
Mandrekar, we will use a generic Hilbert space $H$ and the generic
stability index $p\in (0,2)$. Recall that for a stable random element
$S\in H$ with index $p \in (0,2)$, there exists a {\em spectral}
measure $\sg_S$ defined on the unit sphere $\mathbb{S}_H=\lbr z \in H:
\lnorm z \rnorm=1 \rbr$, such that the characteristic functional of $S$  is
given by
\begin{equation} \label{e:char_fun}
E\exp\{ i \lip x, S \rip \}
=\exp \lbr i \lip x,\beta_S \rip-\int_{\mathbb{S}}| \lip x,s \rip |^p
\sg_S (ds)+i C(p,x)  \rbr, \ \ x \in H,
\end{equation}
where
\[
C(p,x)=
\begin{cases}
\tan \frac{\pi p}{2} \int_{\mathbb{S}}\lip x,s \rip | \lip x,s \rip |^{p-1}
\sg_S (ds)  &\mbox{if $p \neq 1$,}\\
\frac{2}{\pi} \int_{\mathbb{S}}\lip x,s \rip\log  | \lip x,s \rip | \sg_S (ds)  &\mbox{if $p = 1$.}
\end{cases}
\]
We denote the above representation by $S \sim [p,\sg_S,\beta_S]$.  The
$p$-stable random element $S$ is necessarily regularly varying with
index $p\in (0,2)$. In fact, its angular measure is precisely the
normalized spectral measure appearing in \eqref{e:char_fun}, i.e.,
\[
\Gamma_S(\cdot) = \frac{\sigma_S(\cdot)}{\sigma_S({\mathbb S}_H)}.
\]

 \citetext{kuelbs:mandrekar:1974} derived sufficient and necessary
conditions on the distribution of $Z$ under which
\begin{equation}\label{e:KM-stable}
b_N^{-1} \lp \sum_{i=1}^N Z_i-\ga_N  \rp \convd S,
\end{equation}
where the $Z_i$ are iid copies of $Z$. They assume that
the support of the distribution of $S$, equivalently of the distribution
of $Z$, spans the whole Hilbert space $H$. In our context, we will
need to work with $Z$ whose distribution is
not supported on  the whole space.
Denote  by $L(Z)$ the smallest closed subspace which contains
the support of the distribution of $Z$. Then $L(Z)$ is a Hilbert
space itself with the inner product inherited from $H$. Denote by
$\lbr e_j,\ j\in\mathbb N \rbr$ an orthonormal basis of $L(Z)$.
We assume that this is an infinite basis because we consider
infinite dimensional data. (The finite dimensional case has already been dealt
with by \citetext{rvaceva:1962}.)  Introduce the projections
\[
\pi_m(z)=\sum_{j=m}^\infty \lip z, e_j \rip e_j, \ \ z \in H.
\]

\begin{theorem}\label{t:KM}
Let $Z_1$, $Z_2$, $\ldots$ be iid random elements
in a separable Hilbert space $H$ with the same distribution as $Z$.
Let $\lbr e_j,\ j\in\mathbb N \rbr$ be an orthonormal basis of $L(Z)$.
There exist normalizing constants $b_N$ and $\ga_N$ such that
\refeq{KM-stable} holds
{\bf  if and only if}
\begin{equation}\label{e:KM-1}
\frac{P\lp \lnorm \pi_m(Z) \rnorm>tu \rp}{P\lp \lnorm Z \rnorm>u  \rp}
\to \frac{c_m}{c_1}t^{-p}, \ \ u \to \infty,
\end{equation}
where for each $m \ge 1$, $c_m> 0$,
and $\lim_{m \to \infty} c_m=0$,
and where
\begin{equation}\label{e:KM-2}
\frac{P\lp \lnorm Z \rnorm>u, Z/\lnorm Z \rnorm \in A \rp}{P\lp  \lnorm Z \rnorm>u, Z/\lnorm Z \rnorm \in A^\star \rp} \to \frac{\sigma_{S}(A)}{\sigma_{S}(A^\star)}, \ \ u \to \infty,
\end{equation}
for all continuity sets $A$, $A^\star \in \mathcal{B}(\mathbb{S}_H)$
with $\sigma_{S}(A^\star)>0$.

If \refeq{KM-stable} holds, the sequence $b_N$ must satisfy
\begin{equation}\label{e:bN}
b_N \to \infty, \ \ \frac{b_N}{b_{N+1}} \to 1, \ \ Nb_N^{-2}E \lp \lnorm Z \rnorm^2 I_{\lbr \lnorm Z \rnorm \le b_N \rbr} \rp \to  \lambda_p \sigma_S({\mathbb S}_H),
\end{equation}
where
\begin{equation}\label{e:lambda_p}
\lambda_p =  \left\{ \begin{array}{ll}
  \frac{p(1-p)}{\Gamma(3-p) \cos(\pi p/2)} &,\ \mbox{ if }p\not=1\\
  2/\pi &, \mbox{ if } p=1,
  \end{array}\right.
\end{equation}
and $\Gamma(a) := \int_0^\infty e^{-x} x^{a-1} dx, a>0$
 is the Euler gamma function.  Furthermore,
the $\ga_N \in H$ may be chosen as
\begin{equation}\label{e:gaN}
\ga_N =NE \lp ZI_{\lbr \lnorm Z \rnorm\le b_N \rbr} \rp.
\end{equation}
\end{theorem}

\begin{remark} The origin of the constant $\la_p$ appearing in
  \eqref{e:bN} can be understood as follows.  Consider the simple
  scalar case $H=\mathbb R.$ Let $Z$ be symmetric $\alpha$-stable with
  $E [e^{ iZ x}] = e^{ - c |x|^\alpha},\ x\in \mathbb R$, where in
  this case, $c = \sigma({\mathbb S}_H) \equiv \sigma(\{ -1, 1\})> 0$.
  Consider iid copies $Z_i,\ i=1,2,\dots$ of $Z$ and observe that by
  the $p$-stability property
\[
\frac{1}{N^{1/\alpha}} \sum_{j=1}^N Z_j \stackrel{d}{=} Z \equiv S,
\]
and hence \eqref{e:KM-stable} holds trivially with $b_N:=
N^{1/\alpha}$ and $\gamma_N:=0$.

On the other hand, by Proposition 1.2.15 on page 16 in
\citetext{samorodnitsky:taqqu:1994}, we have
\[
P(|Z|>x) \sim  \frac{c(1-p)}{\Gamma(2-p)\cos(\pi p/2)} x^{-p},\ \ \mbox{ as }x\to\infty.
\]
This along with an integration by parts and an application of
Karamata's theorem yield $N b_N^{-2} E [ Z^2 I_{\{ |Z|\le b_N\} }] \to
\lambda_p \sigma_S({\mathbb S}_H)$, giving the constant in
\eqref{e:bN}.
\end{remark}

\begin{proposition}\label{p:KM-equiv}
  Conditions \refeq{KM-1} and \refeq{KM-2} in Theorem~\ref{t:KM} hold
  if and only if $Z$ is regularly varying in $H$ with index $p
  \in (0,2)$ and for each $m\ge 1$, $\mu_Z(A_m) > 0$,
where
\begin{equation}\label{e:Am-def}
A_m= \lbr z \in H: \lnorm \pi_m(z) \rnorm={\Big\|}
\sum_{j=m}^\infty \lip z, e_j \rip e_j {\Big\|} >1 \rbr.
\end{equation}
\end{proposition}

\medskip
Our next objective is to show that if $X$ is a regularly varying
element of a separable Hilbert space $H$ whose index is $\ag>0$,
then the operator $Y=X\otimes X$ is regularly varying with
index $\ag/2$, in the space of Hilbert--Schmidt operators.  If $y, z
\in H$, then $y\otimes z$ is an element of $\cS$ defined by $(y\otimes
z)(x) = \lip y, x\rip z, \ x \in H$. It is easy to check that $\lnorm
y\otimes z \rnorm_\cS = \lnorm y \rnorm \lnorm z \rnorm$.  If $B_1,
B_2 \subset H$, we denote by $B_1\otimes B_2$ the subset of $\cS$
defined as the set of operators of the form $x_1 \otimes x_2$, with
$x_1 \in B_1, x_2 \in B_2$. Denote by ${\mathbb S}_H$ the unit sphere
in $H$ centered at the origin, and by ${\mathbb S}_\cS$ such a sphere
in $\cS$.

\medskip
The next result is valid for all $\alpha>0$.

\begin{proposition}\label{p:Y-RV}
  Suppose $X$ is a regularly varying element with index $\ag>0$ of a
  separable Hilbert space $H$. Then the operator $Y= X\otimes X$ is a
  regularly varying element with index $\alpha/2$ of the space
  $\mathcal{S}$ of Hilbert-Schmidt operators.
 \end{proposition}

\begin{remark} \label{rem:Gamma_Y}
The proof of Proposition \ref{p:Y-RV} shows that
the angular measure of $X\otimes X$ is supported on the {\em diagonal}
$ \lbr \Psi \in {\mathbb S}_\cS :
\Psi = x \otimes x \ {\rm for \ some} \ x \in {\mathbb S}_H \rbr$
and that 
$
\Gamma_{X\otimes X}(B\otimes B) = \Gamma_X(B),\ 
\forall \ B\subset{\cal B}(\mathbb{S}_H).$
\end{remark}

The next result specifies the limit distribution of the sums of the
$X_i\otimes X_i$ based on the results derived so far.

\begin{theorem}\label{t:asym-prodX}
Suppose Assumptions \ref{a:X-RV} and \ref{a:KM} hold.
Then, there exist normalizing constants $k_N$ and operators $\psi_N$
such that
\begin{equation}\label{e:asym-prodX}
k_N^{-1} \lp \sum_{i=1}^N X_i \otimes X_i-\psi_N  \rp \convd S,
\end{equation}
where $S\in \cS$ is a stable random operator,
$S \sim [\ag/2,\sigma_S,0]$, where  the  spectral measure $\sigma_S$ is
defined on the unit sphere $\mathbb{S}_\mathcal{S}=\lbr y \in
\mathcal{S}: \lnorm y \rnorm_\mathcal{S}=1 \rbr$.
The normalizing constants may be chosen as follows
\begin{equation}\label{e:kN}
k_N=\lp \frac{\ag}{4-\ag}  \rp^{2/\ag}a_N^2, 
\ \ \psi_N 
=NE \left[ \lp X \otimes X \rp  I_{\lbr \lnorm X \rnorm^2\le k_N \rbr} \right], 
\end{equation}
where $a_N$ is defined by \refeq{aN}. 
\end{theorem}

The final  result of this section specifies the asymptotic distribution of
$\widehat C - C$.

\begin{theorem}\label{t:asym-hatC-C}
Suppose Assumptions \ref{a:X-RV} and \ref{a:KM} hold.
Then,
\begin{equation}\label{e:limC}
N k_N^{-1} (\widehat C - C) \convd S
- \frac{\alpha}{\alpha-2} \int_{ \displaystyle {\mathbb S}_H } \lp\theta \otimes \theta \rp \Gamma_X (d \theta ),
\end{equation}
where $S\in \cS$ and $\lbr k_N \rbr$ are as
in Theorem~\ref{t:asym-prodX}.
($k_N = N^{2/\ag} L(N)$ for a slowly varying $L$.)
\end{theorem}

If the $X_i$ are scalars, then the angular measure $\Gamma_X$ is
concentrated on $\displaystyle {\mathbb S}_H = \{ -1, 1 \}$,
with $\Gamma_X(1) = p, \Gamma_X(-1) = 1- p$, in the notation
of \citetext{davis:resnick:1986}. Thus
$\int_{{\mathbb S}_H } \theta^2 \Gamma_X(d\theta) = 1$,
and we recover the centering $\ag/(\ag-2)$ in
Theorem 2.2 of \citetext{davis:resnick:1986}.
Relation \refeq{limC} explains the structure of this centering in a much
more general context.

\medskip

Theorem~\ref{t:asym-hatC-C} readily leads  to a strong law of large
numbers which can be derived by an application of the following
result,  a consequence of  Theorem 3.1 of \citetext{acosta:1981}.

\begin{theorem} \label{t:A3.1}
 Suppose $Y_i, i \ge 1,$ are iid mean zero elements of a separable
Hilbert space
with $E \| Y_i \|^\ga < \infty$,  for some $1 \le \ga < 2$. Then,
\[
\frac{1}{N^{1/\ga}} \sum_{i=1}^N Y_i \convP 0 \ \ \
{\it  if \ and \ only \ if } \ \ \
\frac{1}{N^{1/\ga}} \sum_{i=1}^N Y_i \convas 0.
\]
\end{theorem}

Set $Y_i = X_i \otimes X_i - E[X\otimes X]$. Then the $Y_i$ are iid
mean zero elements of $\cS$ which, by Proposition~\ref{p:Y-RV},
satisfy $E\| Y_i \|_\cS^\ga < \infty$, for any $\ga \in (0, \ag/2)$.
Theorem~\ref{t:asym-hatC-C} implies that for any $\ga \in (0, \ag/2)$,
$ N^{-1/\ga} \sum_{i=1}^N Y_i \convP 0 $. Thus Theorem~\ref{t:A3.1}
leads to the following corollary.

\begin{corollary}\label{c:as}
Suppose Assumptions \ref{a:X-RV} and \ref{a:KM} hold.
Then, for any $\ga \in [1, \ag/2)$,
$N^{1-1/\ga} \| \widehat C - C \|_\cS \to 0$ with probability 1.
\end{corollary}

\section{Convergence of eigenfunctions and eigenvalues}
\label{s:KR}
We first formulate and prove
 a general result which allows us to derive the asymptotic
distributions of the eigenfunctions and eigenvalues of an estimator
of the covariance operator from the asymptotic distribution of the operator
itself. The proof of this result is implicit in the proofs of the results
of Section 2 of \citetext{kokoszka:reimherr:2013},  which pertain
to the asymptotic normality of the sample covariance operator if
$E\| X \|^4 < \infty$.
The result and the technique of proof are however more general,
and can be used in different contexts, so we state and prove it in detail.

\begin{assumption} \label{a:ass-conv}
Suppose $C$ is the covariance operator
of a random function $X$ taking values in $L^2$ such that
$E \|X \|^2 < \infty$.  Suppose $\widehat C $ is an estimator of $C$
which is a.s.   symmetric, nonnegative--definite
and  Hilbert--Schmidt.  Assume that for some random operator
$Z\in \cS$, and for some $r_N\to \infty$,
\[
Z_N := r_N(\widehat C - C) \convd Z.
\]
\end{assumption}

In our setting, $Z\in \cS$ is specified in \refeq{limC},
 and $r_N = N^\bg L(N)$
for some $0 < \bg < 1/2$.  More precisely,
\[
 r_N = N a_N^{-2}, \ a_N = N^{1/\ag} L_0(N), \ \ag \in (2,4).
\]

We will work with the eigenfunctions and eigenvalues defined by
\[
C(v_j) = \la_j v_j, \ \  \widehat C_j (\hat v_j) = \hat\la_j \hat v_j, \ \
j \ge 1.
\]
Assumption ~\ref{a:ass-conv} implies that $\hat\la_j \ge 0$ and the
$\hat v_j$ are orthogonal with probability 1.
We assume that, like the  $v_j$, the  $\hat v_j$ have unit norms.
To lighten the notation, we assume that sign$\lip \hat v_j , v_j \rip$ = 1.
This sign does not appear in any of our final results, it  cancels
in the proofs. We assume that both sets of eigenvalues are ordered
in decreasing order. The next assumption is standard, it ensures
that the population  eigenspaces are one dimensional.

\begin{assumption}\label{a:la}
$\la_1 > \la_2, \ldots, > \la_p > \la_{p+1}.$
\end{assumption}

Set
\[
T_j = \sum_{k\neq j} (\la_j - \la_k)^{-1}
\lip Z, v_j \otimes v_k \rip v_k.
\]
Lemma \ref{l:1} in online material 
shows that the series defining $T_j$ converges a.s. in
$L^2$.

\begin{theorem} \label{t:KR1}
Suppose  Assumptions \ref{a:ass-conv} and \ref{a:la} hold. Then,
\[
r_N \lbr \hat v_j - v_j, \ 1 \le j \le p \rbr
\convd \lbr T_j, \ 1 \le j \le p \rbr, \ \  {\rm in}\ (L^2)^p,
\]
and
\[
r_N \lbr \hat\la_j - \la_j, \ 1 \le j \le p \rbr
\convd \lbr \lip Z(v_j), v_j \rip, \ 1 \le j \le p \rbr,
\ \  {\rm in}\ {\mathbb R}^p.
\]
\end{theorem}

If $Z$ is an $(\ag/2)$--stable random
operator in $\cS$, then the  $T_j$ are jointly
 $(\ag/2)$--stable random functions
in $L^2$, and $ \lip Z(v_j), v_j \rip $ are jointly $(\ag/2)$--stable random
variables.  This follows directly from the definition of a stable distribution,
e.g. Section 6.2 of \citetext{linde:1986}.
Under Assumption~\ref{a:X-RV},
$r_N= N^{1-2/\ag} L_0^{-2}(N)$.
Theorem~\ref{t:KR1} thus
leads to the following corollary.

\begin{corollary} \label{c:KR1} Suppose Assumptions \ref{a:X-RV},
\ref{a:KM} and \ref{a:la} hold. Then,
\[
 N^{1-2/\ag} L_0^{-2}(N)
\lbr \hat v_j - v_j, \ 1 \le j \le p \rbr
\convd \lbr T_j, \ 1 \le j \le p \rbr, \ \  {\rm in}\ (L^2)^p,
\]
where the $T_j$ are jointly $(\ag/2)$--stable in $L^2$, and
\[
N^{1-2/\ag} L_0^{-2}(N) \lbr \hat\la_j - \la_j, \ 1 \le j \le p \rbr
\convd \lbr  S_j,  \ 1 \le j \le p \rbr,
\ \  {\rm in}\ {\mathbb R}^p,
\]
where the $S_j$ are jointly $(\ag/2)$--stable in ${\mathbb R}$.
\end{corollary}

Corollary~\ref{c:KR1} implies the rates in probability  $ \hat v_j -
v_j = O_P(r_N^{-1})$ and $\hat\la_j - \la_j =O_P(r_N^{-1})$,  with $r_N=
N^{1-2/\ag} L_0^{-2}(N)$. This means, that the distances between $\hat
v_j$ and $\hat\la_j$ and the corresponding population parameters are
approximately of the order $N^{2/\ag - 1}$, i.e. are asymptotically
larger that these distances in the case of $E\|X \|^4 < \infty$, which
are of the order $N^{-1/2}$. Note that $
2/\ag - 1 \to -1/2$,  as  $\ag \to 4$.

It is often useful to have some bounds on moments, analogous to
relations \refeq{rate}.  Since the tails of $\| T_j \|$ and $|S_j|$
behave like $t^{-\ag/2}$, e.g. Section 6.7 of \citetext{linde:1986},
$E \| T_j \|^\gamma < \infty, \ 0 < \ga < \ag/2$, with an analogous
relation for $|S_j |$. We can thus expect convergence of moments of
order $\ga \in (0, \ag/2)$.  The following theorem specifies the
corresponding results.

\begin{theorem} \label{t:M}
If  Assumptions \ref{a:X-RV} and \ref{a:KM} hold, then for each
$\ga \in (0, \ag/2)$, there is a slowly varying function $L_\ga$ such that
\[
\limsup_{N\to \infty}N^{\ga(1-2/\ag)} L_\ga(N)
E \lnorm \widehat C - C \rnorm_\cS^\ga < \infty
\]
and for $j \ge 1$,
\[
\limsup_{N\to \infty}N^{\ga(1-2/\ag)} L_\ga(N)
E | \hat\la_j - \la_j |^\ga< \infty.
\]
If, in addition, Assumption~\ref{a:la} holds, then for $1 \le j \le p$,
\[
\limsup_{N\to \infty}N^{\ga(1-2/\ag)} L_\ga(N)
E \lnorm \hat v_j - v_j \rnorm^\ga < \infty.
\]
\end{theorem}

Several cruder bounds can be derived from Theorem~\ref{t:M}. In
applications, it is often convenient  to take $\ga=1$. Then
$E \| \widehat C - C \|_\cS \le N^{2/\ag -1} L_1(N)$. By Potter bounds,
e.g. Proposition 2.6 (ii) in \citetext{resnick:2006}, for any $\epsilon> 0$
there is a constant $C_\epsilon$ such that for $x> x_\epsilon$
$L_1(x) \le C_\epsilon x^\epsilon$. For each $\ag \in (2, 4)$, we can
choose $\epsilon$  so small that $-\dg(\ag) := 2/\ag - 1 + \epsilon < 0$.
This leads to the following corollary.

\begin{corollary} \label{c:M} If  Assumptions  \ref{a:X-RV} and \ref{a:KM}
hold, then for each $\ag \in (2, 4)$, there are constant
$C_\ag$ and $\dg(\ag) > 0$ such that
\[
E \| \widehat C - C \|_\cS \le C_\ag N^{-\dg(\ag)}
\ \ \ {\rm and} \ \ \
E \| \hat\la_j- \la_j \| \le C_\ag N^{-\dg(\ag)}.
\]
If, in addition, Assumption~\ref{a:la} holds, then for $1 \le j \le p$,
$E \lnorm \hat v_j - v_j \rnorm \le C_\ag(j) N^{-\dg(\ag)}$.
\end{corollary}
Corollary \ref{c:M} implies that $E \| \widehat C - C \|_\cS$,
$E \| \hat\la_j- \la_j \|$ and $E \lnorm \hat v_j - v_j \rnorm$ tend to
zero,  for any  $\ag \in (2, 4)$.

\section{An application: functional linear regression} \label{s:app}
One of the most widely used tools of functional data analysis is the
functional regression model, e.g. \citetext{ramsay:silverman:2005},
\citetext{HKbook}, \citetext{KRbook}.  Suppose $X_1, X_2, \ldots, X_N$
are explanatory functions, $Y_1, Y_2, \ldots, Y_N$ are response functions,
and assume that
\begin{equation} \label{e:flr}
Y_i(t) = \int_0^1 \psi(t,s) X_i(s) ds + \eg_i(t), \ \ \ 1 \le i \le N,
\end{equation}
where $\psi(\cdot, \cdot)$ is the kernel of $\Psi \in \cS$.  The $X_i$
are mean zero iid functions in $L^2 = L^2([0,1])$, and so are the
error functions $\eg_i$. Consequently, the $Y_i$ are iid in $L^2$.  A
question that has been investigated from many angles is how to
consistently estimate the regression kernel $\psi(\cdot, \cdot)$. An
estimator that has become popular following the work of
\citetext{yao:muller:wang:2005AS} can be constructed as follows.

The population version of \refeq{flr} is
$Y(t) = \int \psi(t,s) X(s) ds + \eg(t)$.
Denote by $v_i$ the FPCs of $X$ and by $u_j$ those of $Y$, so that
\[
X(s) = \sum_{i=1}^\infty \xi_i v_i(s), \ \ \
Y(t) = \sum_{j=1}^\infty \zg_j u_j(t).
\]
If $\eg$ is independent of $X$, then, with $\la_\ell = E[\xi_\ell^2]$,
\[
\psi(t,s) = \sum_{k=1}^\infty \sum_{\ell =1}^\infty
\frac{E[\xi_\ell \zg_k]}{\la_\ell} u_k(t) v_\ell(s),
\]
with the series converging in $L^2([0,1]\times [0,1])$, equivalently
in $\cS$,
see Lemma  8.1 in \citetext{HKbook}. This motivates the  estimator
\[
\hat\psi_{KL}(t,s) =  \sum_{k=1}^K \sum_{\ell =1}^L
\frac{\hat\sg_{\ell k}}{\hat\la_\ell}  \hat u_k(t) \hat v_\ell(s),
\]
where $\hat u_k$ are the eigenfunctions of $\widehat C_Y$ and $\hat\sg_{\ell k}$ is an estimator of $E[\xi_\ell \zg_k]$.
\citetext{yao:muller:wang:2005AS} study the above estimator
under the assumption that data are observed sparsely and with
measurement errors. This requires two-stage smoothing, so their
assumptions focus on conditions on the various smoothing parameters
and the random mechanism that generates the sparse observations.
Like in all work of this type, they  assume that the underlying functions
have finite fourth moments: $E \| X \|^4 < \infty$,
$E \| \eg  \|^4 < \infty$, and so  $E \| Y \|^4 < \infty$.
Our objective is to show that if the $X_i$ satisfy the assumptions
of Section~\ref{s:prelim}, then
\begin{equation} \label{e:con-Y}
\lnorm \widehat\Psi_{KL} - \Psi \rnorm_\cL \convas 0,
\end{equation}
as $N\to \infty$, and $K, L \to \infty$ at suitable rates determined
by the rate of decay of the eigenvalues.  The norm $\| \cdot \|_\cL$
is the usual operator norm.
The integral operators
$\Psi$ and $\widehat\Psi_{KL}$ are defined by their kernels
$\psi(\cdot, \cdot)$ and $\hat\psi_{KL}(\cdot, \cdot)$, respectively.
 We focus on
moment conditions, so we assume that the functions $X_i, Y_i$ are
fully observed, and use the estimator
\[
\hat\sg_{\ell k} = \frac{1}{N} \sum_{i=1}^N
\hat\xi_{i\ell} \hat\zg_{ik}, \ \ \
\hat\xi_{i\ell} = \lip X_i, \hat v_\ell \rip, \
\hat\zg_{ik} = \lip Y_i, \hat u_k \rip.
\]
Since the regression operator $\Psi$ is infinitely dimensional,
we strengthen Assumption~\ref{a:la} to the following assumption.
\begin{assumption} \label{a:diff} The eigenvalues
$\la_i = E\xi_i^2$ and $\ga_j = E\zg_j^2$ satisfy
\[
\la_1 > \la_2 > \ldots > 0, \ \ \ \ga_1 > \ga_2 > \ldots > 0.
\]
\end{assumption}

Many issues related to the infinite dimension of the functional data
in model \refeq{flr} are already present when considering projections
on the unobservable subspaces
\[
{\mathcal V}_L = {\rm span}\lbr v_1, v_2, \ldots, v_L \rbr,
\ \ \ {\mathcal U}_K = {\rm span} \lbr u_1, u_2, \ldots, u_K \rbr.
\]
Therefore we first consider the convergence of the
operator with the kernel
\[
\psi_{K L} (t,s) = \sum_{k=1}^K  \sum_{\ell =1}^L
\frac{\sg_{\ell k}}{\la_\ell} u_k(t) v_\ell(s).
\]
Set
$\sg_{\ell k} = E[\xi_\ell \zg_k]$ and observe that
\[
\psi_{K L}(t,s) - \psi(t, s) = - \sum\limits_{k> K \ {\rm or} \ \ell > L}
\frac{\sg_{\ell k}}{\la_\ell} u_k(t) v_\ell(s).
\]
Therefore
\begin{equation} \label{e:L-S}
\lnorm \Psi_{KL} - \Psi \rnorm_\cL^2\le
\lnorm \Psi_{KL} - \Psi \rnorm_\cS^2
=  \sum\limits_{k> K \ {\rm or} \ \ell > L}
\frac{\sg^2_{\ell k}}{\la^2_\ell}.
\end{equation}
The condition
\begin{equation} \label{e:A1}
\sum_{k=1}^\infty  \sum_{\ell =1}^\infty
\frac{\sg^2_{\ell k}}{\la^2_\ell} < \infty,
\end{equation}
which is Assumption (A1) of \citetext{yao:muller:wang:2005AS},
implies that the remainder term is asymptotically negligible.
It is instructive to rewrite condition \refeq{A1} in a different form.
Observe that
\begin{equation} \label{e:sg-lk}
\sg_{\ell k}
= E [ \xi_l \lip \Psi(X) + \eg, u_k \rip ]
= E[ \xi_l \sum_{i=1}^\infty \xi_i \lip \Psi(v_i), u_k \rip ]
=\la_\ell \lip \Psi(v_\ell), u_k \rip.
\end{equation}
Therefore
\begin{equation} \label{e:H2}
\sum_{k=1}^\infty  \sum_{\ell =1}^\infty
\frac{\sg^2_{\ell k}}{\la^2_\ell}
= \sum_{\ell =1}^\infty \frac{1}{\la^2_\ell} \sum_{k=1}^\infty
\la_\ell^2 \lip \Psi(v_\ell), u_k \rip^2
=  \sum_{\ell =1}^\infty \lnorm \Psi(v_\ell) \rnorm^2
= \lnorm \Psi \rnorm_\cS^2.
\end{equation}
We see that condition \refeq{A1} simply means that $\Psi$ is a
Hilbert--Schmidt operator, and so it holds under our general assumptions
on model \refeq{flr}.

The last  assumption implicitly restricts  the rates at which
$K$ and $L$ tend to infinity with $N$. Under Assumption~\ref{a:diff},
the following quantities are well defined
\begin{equation}\label{e:ag-j}
\ag_j = \min\lbr \la_{j} - \la_{j+1}, \la_{j-1} - \la_j \rbr, \ j \ge 2,
\ \ \ \ag_1 = \la_1 - \la_2,
\end{equation}
\begin{equation}\label{e:bg-j}
\bg_j = \min\lbr \ga_{j} - \ga_{j+1}, \ga_{j-1} - \ga_j \rbr, \ j \ge 2,
\ \ \ \bg_1 = \ga_1 - \ga_2.
\end{equation}

\begin{assumption} \label{a:rates} The truncation levels
$K$ and $L$ tend to infinity with $N$ in such a way that
for some $\ga \in (1, \ag/2)$,
\begin{align}
&
\limsup_{N\to \infty} \la_L^{-3/2} L^{1/2} N^{1/\ga-1}< \infty,
\label{e:c-a}\\
&
\limsup_{N\to \infty} \la_L^{-1}
\lp \sum_{j=1}^L  \ag_j^{-1} \rp N^{1/\ga-1}< \infty,
\label{e:c-b}\\
&
 \limsup_{N\to \infty} \la_L^{-1} K^{1/2}N^{1/\ga-1}< \infty,
\label{e:c-?} \\
&
\limsup_{N\to \infty}
\la_L^{-1} \lbr \lp \sum_{k=1}^K \bg_k^{-1} \rp
+ \lp \sum_{k=1}^K \bg_k^{-2} \rp^{1/2} \rbr
N^{1/\ga-1}< \infty.
\label{e:c-D}
\end{align}
\end{assumption}

The conditions in Assumption~\ref{a:rates} could be restated or unified;
and could be replaced by slightly different conditions by modifying the
technique of proof. The essence of this assumption is that $K$ and $L$
must tend to infinity sufficiently slowly, and the rate is influenced by index
$\ag$;
the closer  $\ag$ is to 4, the larger $\ga$
can be taken, so $K$ and $L$ can be larger.

\medskip

\begin{theorem} \label{t:a}
Suppose model \refeq{flr} holds with $\Psi\in \cS$,
 the $X_i$ and the $Y_i$
satisfying Assumptions~\ref{a:X-RV} and \ref{a:KM},
 and square integrable $\eg_i$, $E \lnorm \eg_i \rnorm^2 < \infty$. Then
relation \refeq{con-Y} holds
under Assumptions \ref{a:diff} and \ref{a:rates}.
\end{theorem}

\bigskip

\noindent{\bf Acknowledgements:}
The authors have been supported by NSF grant
``FRG: Collaborative Research: Extreme Value
Theory for Spatially Indexed Functional Data''
(1462067 CSU, 1462368 Michigan).
We thank Professor Jan Rosi{\'n}ski for
directing us to the work of \citetext{acosta:gine:1979},
and Mr. Ben Zheng for preparing the figures.

\bigskip

\bibliographystyle{oxford3}
\bibliography{qian}

\newpage

\centerline{\LARGE \it Online  material}

\bigskip \bigskip


\section{Proofs of the results stated in the paper} \label{s:p}

Throughout the proofs, we will use relatively well--known properties
of slowly varying functions, which we collect in
Lemma~\ref{l:L-properties} for ease of reference.
For the proofs and many more details, see e.g., \citetext{resnick:1987} and
\citetext{bingham:goldie:teugels:1987}.

\begin{lemma}\label{l:L-properties}
If $L$ is a slowly varying function, then:

(i) $L_1(u) = L(u^\rho),\ \rho>0$ and $L_2(u) = |L(u)|^a,\ a\in\mathbb R$ are slowly varying.

(ii) {\em (Potter bounds)} For all $\delta>0$, we have
$L(u) = o(u^{\delta}),$ as $u\to\infty$.

(iii) {\em (Karamata's Theorem)} For all $\rho> -1$ and $\eta>1$, as $u\to\infty$, we have
$$
\int_{0}^u x^{\rho} L(x) dx \sim \frac{u^{\rho+1} L(u)}{(\rho+1)}\quad \mbox{ and }\quad  \int_{u}^\infty x^{-\eta} L(x) dx \sim \frac{u^{-(\eta-1)} L(u)}{(\eta -1)},
$$
where $a(u) \sim b(u)$ means $a(u)/b(u) \to 1$, as $u\to\infty$.
\end{lemma}

\subsection{Proofs of Proposition~\ref{p:HS-as} and of
the results of Section~\ref{s:hatC}}

\subsubsection*{Proof of Proposition~\ref{p:HS-as}}
Since $C$ is a covariance operator, it is nuclear
($\sum_{j\ge 1 } \la_j < \infty$), e.g. Theorem 11.2.2 of
\citetext{KRbook}, and so it is Hilbert--Schmidt
($\sum_{j\ge 1 } \la_j^2 < \infty$).

We now verify that $\widehat C$ is a.s. a Hilbert-Schmidt
operator. Observe that
\[
\|\widehat C\|_\mathcal{S}^2
= \iint \hat c^2(t,s) dt ds
= \iint \lbr \frac{1}{N} \sum_{n=1}^N X_n(t) X_n(s) \rbr^2 dt ds.
\]
It thus suffices to show that
\[
\iint \lbr X_n(t) X_n(s) \rbr^2 dtds
= \iint S_n^2(t,s) dtds < \infty \ \ \ a.s.,
\]
where
\[
S_n(t,s) = \sum_{j=1}^\infty \xi_{nj} v_j(t)
\sum_{\jp=1}^\infty \xi_{n\jp} v_{\jp}(s).
\]
Observe that
\[
\iint S_n^2(t,s)  dt ds =
\sum_{j, \jp =1}^\infty \sum_{i, \ip = 1}^\infty
\xi_{nj} \xi_{n\jp} \xi_{ni} \xi_{\ip}
\int v_j(t) v_i(t) dt \int v_{\jp}(s) v_{\ip} (s) ds.
\]
Therefore, by the orthonormality by the $v_j$,
\[
\iint S_n^2(t,s) dt ds
= \sum_{j, \jp =1}^\infty \xi_{nj} \xi_{n\jp} \xi_{nj} \xi_{n\jp}
= \lbr \sum_{j=1}^\infty \xi_{nj}^2 \rbr^2.
\]
Finally, observe that
\[
\sum_{j=1}^\infty \xi_{nj}^2 = \int_0^1 X_n^2(t)dt
= \| X\|^2 < \infty \ \ \ a.s.
\]
because $X$ is a random element of $L^2$.

\subsubsection*{Proof of Proposition \ref{p:KM-equiv}}
Set
\begin{equation} \label{e:Ga}
\Gamma(\cdot)
=\frac{\sigma_{S}(\cdot)}{\sigma_{S}(\mathbb{S}_H)}
\end{equation}
Recall that \refeq{Ga} specifies the relationship between the stable
spectral measure $\sigma_S$ and the angular measure $\Gamma$
of a regularly varying distribution appearing in Proposition~\ref{p:RV}.

 First we assume \refeq{KM-1} and \refeq{KM-2} hold. Take $m=1$ in
 \refeq{KM-1} and $A^\star=\mathbb{S}_H$ in \refeq{KM-2}, we then have
 for every $t>0$,
\[
\begin{split}
\frac{P(\lnorm Z \rnorm>tu, Z/\lnorm Z \rnorm  \in A )}{P(\lnorm Z \rnorm>u)} &=\frac{P(\lnorm Z \rnorm>tu, Z/\lnorm Z \rnorm  \in A )}{P(\lnorm Z \rnorm>tu, Z/\lnorm Z \rnorm  \in \mathbb{S}_H )} \frac{P\lp \lnorm Z \rnorm>tu  \rp}{P\lp \lnorm Z \rnorm>u \rp}\\
& \to \frac{\sigma_{S}(A)}{\sigma_{S}(\mathbb{S}_H)} t^{-p} \ \ (u \to \infty) \\
&=\Gamma(A)t^{-p},
\end{split}
\]
for any continuity set $A$ of $\sigma_{S}$ (equivalently, of
$\Gamma$). Thus condition (ii) in Proposition~\ref{p:RV} holds,
which implies that $Z$ is regularly varying with index $p$.

Next we assume that $Z$ is regularly varying with index $
p$, and show that \refeq{KM-1} and \refeq{KM-2} will hold.
Using condition (ii) in Proposition~\ref{p:RV}, we have
\[
\begin{split}
\frac{P\lp \lnorm Z \rnorm>u, Z/\lnorm Z \rnorm \in A \rp}{P\lp  \lnorm Z \rnorm>u, Z/\lnorm Z \rnorm \in A^\star \rp}
 &=\frac{P(\lnorm Z \rnorm>u, Z/\lnorm Z \rnorm  \in A )}{P\lp \lnorm Z \rnorm>u  \rp} \frac{P\lp \lnorm Z \rnorm>u  \rp}{P\lp  \lnorm Z \rnorm>u, Z/\lnorm Z \rnorm \in A^\star \rp}\\
 & \to \frac{\Gamma(A)}{\Gamma(A^\star)} =\frac{\sigma_{S}(A)}{\sigma_{S}(A^\star)},  \ \ (u \to \infty)
\end{split}
\]
for all continuity sets $A$, $A^\star \in \mathcal{B}(\mathbb{S}_H)$
with $\sigma_{S}(A^\star)>0$.
 Then, with the set $A_m$ defined by \refeq{Am-def},
\[
\begin{split}
\frac{P\lp \lnorm \pi_m(Z) \rnorm>tu \rp}{P\lp \lnorm Z \rnorm>u  \rp}&=\frac{P(t^{-1}u^{-1}Z \in  A_m )}{P(u^{-1}Z \in  A_1 )}\\
&=\frac{P(\lnorm Z \rnorm>u)}{P(u^{-1}Z \in  A_1 )} \frac{P(t^{-1}u^{-1}Z \in  A_m )}{P(\lnorm Z \rnorm>tu)} \frac{P\lp \lnorm Z \rnorm>tu \rp}{P\lp \lnorm Z \rnorm>u  \rp}\\
&\to \frac{\mu_Z(A_m)}{\mu_Z(A_1)}t^{-p} =:\frac{c_m}{c_1}t^{-p}, \ \ (u \to \infty)
\end{split}
\]
where the above convergence follows from \eqref{e:X-RV-mu} provided we can show that $A_m,\ m\ge 1$ are continuity sets of the measure $\mu_Z$.  We do that next.

By the definition of $A_m$ in \eqref{e:Am-def} and since $\pi_m$ is continuous and homogeneous,
we have
$$
 \partial A_m=\lbr z \in H: \lnorm \pi_m(z) \rnorm=1 \rbr\quad\mbox{ and }\quad \partial (rA_m)=r\partial A_m=\lbr z \in H: \lnorm \pi_m(z) \rnorm=r \rbr.
 $$
 Furthermore, we have that $r_1A_m \supset r_2 A_m$ for all $0<r_1<r_2$. This implies that $A_m=\cup_{r>1}\partial (rA_m)$, where the sets $\partial (rA_m)$ are all disjoint in $r$. By the homogeneity of $\mu_Z$, however,
 (recall \eqref{e:mu-prop}) it follows that $\mu_Z(\partial (rA_m))=r^{-p}\mu_Z(\partial A_m)$. In particular,
\[
\mu_Z(A_m) \ge \sum_i \mu_Z(\partial (r_iA_m))=\sum_i r_i^{-p}\mu_Z(\partial A_m),
\]
for any sequence $r_i>1$. If $\mu_Z(\partial A_m)>0$, then by taking $r_i$'s such that $\sum_i r_i^{-p}=\infty$, we obtain $\mu_Z(A_m)=\infty$,
which is not possible since $A_m$ is bounded away from zero. We have thus shown that $\mu_Z(\partial A_m)=0$, i.e., $A_m$ is a continuity set
of $\mu_Z$ for all $m\ge 1$.

To complete the proof of \refeq{KM-1}, it remains is to show that
$c_m=\mu_Z(A_m) \to 0$, as $m \to \infty$. Notice that $A_m \supset
A_{m+1}$ and thus $\lim_{m \to \infty}
\mu_Z(A_m)=\mu_Z(\cap_{m=1}^\infty A_m),$ since
$\mu(A_1)<\infty$.  It
is easy to see that $\cap_{m=1}^\infty A_m=\emptyset$. Indeed, for
each $z\in H$, we have $\|z\|^2 = \sum_{j=1}^\infty \lip z, e_j \rip^2
<\infty$ and therefore
\[
\|\pi_m(z)\|^2 = \sum_{j=m}^\infty \lip z, e_j \rip^2 \to 0,
\quad \mbox{ as }m\to\infty.
\]
If $z\in \cap_{m\ge 1}A_m$, then $\lnorm \pi_m(z) \rnorm>1$ for each
$m\ge 1$, which is impossible.

\subsubsection*{Proof of Proposition~\ref{p:Y-RV}}
Since $\lnorm Y \rnorm_{\cS} = \lnorm X \rnorm^2$ and
$P(\lnorm X \rnorm > u ) = u^{-\ag} L(u)$, we  conclude that
\[
P(\lnorm Y \rnorm_\mathcal{S} > u)
=u^{-\ag/2} L(u^{1/2}).
\]
Notice that $u\mapsto L(u^{1/2})$ is a slowly varying function.  Thus,
by Proposition \ref{p:RV} (iii), to establish the regular variation of
$Y$ it remains to show that there must exist a probability measure
$\Gamma_Y$ on ${\mathbb S}_\cS$ such that
\[
P \lp \lnorm Y \rnorm_\mathcal{S}^{-1} Y \in A |
\lnorm Y \rnorm_\mathcal{S}>u  \rp \to  \Gamma_Y(A), \ \ u \to \infty,
\]
for every $\Gamma_Y$-continuity set $A$.
The operator $Y$ takes values only in a small subset of
${\mathbb S}_\cS$, namely in
\begin{equation}\label{e:Y-support}
{\mathbb S}_\cS(1) =
\lbr \Psi \in {\mathbb S}_\cS :
\Psi = x \otimes x \ {\rm for \ some} \ x \in {\mathbb S}_H \rbr.
\end{equation}
The set  ${\mathbb S}_\cS(1)$ is closed in ${\mathbb S}_\cS$ and its
Borel subsets have the form $B\otimes B$, where $B$ is a Borel subset
of ${\mathbb S}_H$.  We know that
\[
\Gamma^{(u)}(B)
:=P \lp X/ \lnorm X \rnorm \in B | \lnorm X \rnorm >u  \rp \to
\Gamma(B), \ \ u \to \infty,
\]
for every $\Gamma$-continuity set $B \in \mathbb{S}_H$. Denote by
$\xi_u$ a random element of $H$ taking values in $\mathbb{S}_H$ whose
distribution is $\Gamma^{(u)}$. Then we have
\begin{equation}\label{e:asy-xi}
\xi_u \convd \xi, \ \ u \to \infty,
\end{equation}
where $\xi$ has distribution $\Gamma$. Furthermore, denote by $\eta_u$
a random element of $\cS$ taking values in ${\mathbb S}_\cS(1)$
whose distribution is
\[
P \lp \eta_u \in A  \rp=\frac{P \lp \lnorm Y \rnorm_\mathcal{S}^{-1}
 Y \in A,  \lnorm Y \rnorm_\mathcal{S}>u  \rp}
{P \lp  \lnorm Y \rnorm_\mathcal{S}>u  \rp}, \ \ A \in {\mathbb S}_\cS(1).
\]
We want to identify a random element $\eta$ such that
\begin{equation}\label{e:asy-eta}
\eta_u \convd \eta, \ \ u \to \infty,
\end{equation}
whose distribution will be the desired measure $\Gamma_Y$.

We first verify that
\begin{equation}\label{e:eta-xi}
\eta_u \stackrel{d}{=} \xi_{u^{1/2}} \otimes \xi_{u^{1/2}}.
\end{equation}
Relation \refeq{eta-xi} is equivalent to
\begin{equation} \label{e:eta-xi-dist}
\frac{P \lp \lnorm Y \rnorm_\mathcal{S}^{-1} Y \in A,
\lnorm Y \rnorm_\mathcal{S}>u  \rp}{P \lp  \lnorm Y \rnorm_\mathcal{S}>u  \rp}=P \lp \xi_{u^{1/2}} \otimes \xi_{u^{1/2}} \in A  \rp,
\ \ \forall A \in {\mathbb S}_\cS(1).
\end{equation}
Set $A=B\otimes B$.
Since $\lnorm Y \rnorm_\mathcal{S}=\lnorm X \rnorm^2$, the left--hand
side of \refeq{eta-xi-dist} is
\[
\begin{split}
\frac{P \lp \lnorm Y \rnorm_\mathcal{S}^{-1} Y \in A,  \lnorm Y \rnorm_\mathcal{S}>u  \rp}{P \lp  \lnorm Y \rnorm_\mathcal{S}>u  \rp}
&=\frac{P \lp \lp \lnorm X \rnorm^{-1} X \rp \otimes \lp \lnorm X \rnorm^{-1} X \rp \in B \otimes B,  \lnorm X \rnorm>u^{1/2}  \rp}{P \lp  \lnorm X \rnorm>u^{1/2}  \rp}\\
&=\frac{P \lp \lnorm X \rnorm^{-1} X \in B, \lnorm X \rnorm >u^{1/2}  \rp}{P \lp  \lnorm X \rnorm>u^{1/2}  \rp}\\
&=\Gamma^{(u^{1/2})}\lp B  \rp,
\end{split}
\]
while the right--hand side of \refeq{eta-xi-dist} is
\begin{equation} \label{e:AB}
P \lp \xi_{u^{1/2}} \otimes \xi_{u^{1/2}} \in A  \rp
=P \lp \xi_{u^{1/2}} \in B, \xi_{u^{1/2}} \in B  \rp
=P \lp \xi_{u^{1/2}} \in B \rp=\Gamma^{(u^{1/2})}\lp B  \rp.
\end{equation}
Therefore,  \refeq{eta-xi} holds. It remains to  show that
\[
\eta_u \stackrel{d}{=} \xi_{u^{1/2}} \otimes \xi_{u^{1/2}}
\convd \xi \otimes \xi =:\eta, \ \ u \to \infty.
\]
The above relation  holds because by \refeq{AB}
and  \refeq{asy-xi},
\[
P \lp \xi_{u^{1/2}} \otimes \xi_{u^{1/2}} \in A  \rp
=\Gamma^{(u^{1/2})}\lp B  \rp
 \to \Gamma(B)
=P \lp \xi \in B  \rp
=P \lp \eta \in A  \rp,
\]
provided  $B$ is a continuity set of $\Gamma$. Using the relation
$\lnorm
y\otimes z \rnorm_\cS = \lnorm y \rnorm \lnorm z \rnorm$,
 it is easy to check that $x_n\otimes x_n \to x \otimes x$ in $\cS$
if and only if $x_n\to x$ in $H$. Hence,
$\partial A = \partial B \otimes \partial B$, so the continuity sets
of the distribution of $\eta$ have the form $ B \otimes  B$ with
$\Gamma(\partial B) = 0$.

\subsubsection*{Proof of Theorem~\ref{t:asym-prodX}}
By Proposition~\ref{p:Y-RV}, the operators $X_i\otimes X_i$ are
iid regularly varying elements of $\cS$, whose index of regular
variation is $\ag/2 \in (1,2)$. In order to use Theorem~\ref{t:KM},
we first verify that
$\mu_{X\otimes X}(A_m) > 0$, cf. Proposition~\ref{p:KM-equiv}.
This is where Assumption \ref{a:KM} comes into play.
An orthonormal basis of $L(X\otimes X)$ is
$\lbr v_i \otimes v_j, i, j \ge 1\rbr$, where the $v_j$ are the FPCs of
$X$. Set
\[
A_{n, m} =
\lbr \Psi\in \cS: \ \lnorm \sum_{i=n}^\infty \sum_{j=m}^\infty
\lip \Psi, v_i\otimes v_j \rip_{\cS} v_i\otimes v_j \rnorm_\cS > 1 \rbr.
\]
We must thus verify that
$\mu_{X\otimes X}(A_{n,m}) > 0$. By \refeq{X-RV-mu},
\[
\mu_{X\otimes X}(A_{n,m})
= \lim_{u\to \infty}
\frac{P(X\otimes X \in u A_{n,m})}{P(\|X \otimes X \|_\cS > u)}.
\]
Clearly
\[
P(\|X \otimes X \|_\cS > u) = P(\|X \|^2 > u)
= P\lp \sum_{j=1}^\infty \xi_j^2 > u\rp,
\]
which is the denominator of $Q_{nm}$ in Assumption~\ref{a:KM}.
Turning to the numerator, observe that $X\otimes X \in u A_{nm}$
iff
\[
\lnorm \sum_{i=n}^\infty \sum_{j=m}^\infty
\lip X\otimes X, v_i\otimes v_j \rip_{\cS} v_i\otimes v_j  \rnorm_{\cS}
> u.
\]
Direct verification, which uses the definition of the inner product in $\cS$
and the orthonormality of the $v_j$,  shows that
$
\lip X\otimes X, v_i\otimes v_j \rip_{\cS} = \xi_i\xi_j.
$ It follows that  $X\otimes X \in u A_{nm}$ iff
\[
\lnorm \sum_{i=n}^\infty \sum_{j=m}^\infty
\xi_i \xi_j v_i \otimes v_j \rnorm_\cS^2 > u^2.
\]
Using the definition of the Hilbert--Schmidt norm and the orthogonality
of the $v_j$ again, we see that the above inequality is equivalent to
$
\sum_{i=n}^\infty \xi_i^2 \sum_{j=m}^\infty \xi_j^2 > u^2,
$
so $P(X\otimes X \in u A_{nm})$ is equal to the numerator of $Q_{nm}$.

It remains to show that the normalizing sequences can be chosen
as specified in \refeq{kN}.
It is easy to check that $k_N \to \infty$ and $\frac{k_N}{k_{N+1}} \to 1$.  We will show that
\begin{equation}\label{e:kN-verify}
Nk_N^{-2}E \lp \lnorm X \rnorm^4I_{\lbr \lnorm X \rnorm^2\le k_N \rbr}
\rp \to 1,
\end{equation}
which in view of \eqref{e:bN} would yield \eqref{e:asym-prodX}, where the spectral measure of the limit $S$ is normalized so that
$\lambda_p\sigma_{S}(\mathbb{S}_\mathcal{S}) = 1$ with $\lambda_p$ in \eqref{e:lambda_p}.

Observe that by the Tonelli-Fubini Theorem, we have
\begin{align*}
E \left[   \lnorm X \rnorm^4I_{\lbr \lnorm X \rnorm^2\le k_N \rbr} \right] &= E \left[ \int_0^\infty I_{\{ x< \|X\|^4 \le k_N^2 \} } dx\right]\\
& = \int_0^{k_N^2}  \left[ P ( \| X\|^4 >x) - P(\|X\|^2 > k_N) \right] dx \\
& = \int_0^{k_N^2} x^{-\alpha/4} L(x^{1/4}) dx - k_N^2 k_N ^{-\alpha/2} L(k_N^{1/2}),
\end{align*}
where we used the fact that $P(\|X\|>x) = x^{-\alpha} L(x)$.  Now, by applying Karamata's theorem (Lemma \ref{l:L-properties} (iii))
to the integral in the last expression, we obtain
\begin{align}\label{e:Karamata-once}
E \left[   \lnorm X \rnorm^4I_{\lbr \lnorm X \rnorm^2\le k_N \rbr} \right]
&\sim \frac{1}{(1-\alpha/4)} k_N^{2-\alpha/2} L(k_N^{1/2}) -k_N^ {2-\alpha/2} L(k_N^{1/2}) \nonumber \\
&= \lp \frac{4}{(4-\alpha)} - 1\rp k_N^{2-\alpha/2} L(k_N^{1/2}) \nonumber\\
& = \frac{\alpha}{(4-\alpha)}k_N^{2-\alpha/2} L(k_N^{1/2}),
\end{align}
as $k_N\to\infty$, where $c_N \sim d_N$ means that $c_N/d_N \to 1$.

In view of \eqref{e:covp_x} by taking $A = \{x\, :\, \|x\|> 1\}$, we obtain
\begin{equation}\label{e:aN_and_L}
N P(\|X\| >a_N) = N a_N^{-\alpha} L(a_N) \to 1,
\end{equation}
since $\mu$ is normalized so that $\mu(A)=1$ and
$\mu(\partial A) = 0$ by Proposition 2.2.2 of \citetext{meinguet:2010}.
Thus, multiplying \eqref{e:Karamata-once} by $N k_N^{-2}$ and recalling
\eqref{e:kN}, we obtain
\[
Nk_N^{-2}E \lp \lnorm X \rnorm^4I_{\lbr \lnorm X \rnorm^2\le k_N \rbr} \rp \sim c_\alpha^\alpha k_N^{-\alpha/2} L( k_N^{1/2})  = a_N^{-\alpha} L( c_\alpha a_N),
\]
where $c_\alpha = (\alpha/(4-\alpha))^{1/\alpha}$.  Since $L$ is a
slowly varying function, we have $L(c_\alpha a_N) \sim L(a_N)$ as
$a_N\to \infty$, and therefore by \eqref{e:aN_and_L}, we obtain
\eqref{e:kN-verify}.  This completes the proof.

\subsubsection*{Proof of Theorem~\ref{t:asym-hatC-C}}
Observe that by \refeq{hatC-C},
\begin{equation} \label{e:S1}
N k_N^{-1} \lp \widehat C - C \rp
=  k_N^{-1} \lp \sum_{n=1}^N X_n \otimes X_n - \psi_N \rp
+  k_N^{-1} N E \lb (X\otimes X) I_{\lbr \| X \|^2 > k_N \rbr} \rb,
\end{equation}
with $k_N$ and $\psi_N$ as in Theorem~\ref{t:asym-prodX}.  The first
term converges to $S$, so we must verify the existence of the second
term, show that it converges, and describe its limit.  The issue is
subtle because $k_N\to \infty$ implies that
$k_N^{-1}N \lb (X\otimes X) I_{\lbr
  \| X \|^2 \ge k_N \rbr} \rb \to 0$ with probability 1, yet the
expected value does not tend to zero even in the case of scalar
observations, see Theorem 2.2 of \citetext{davis:resnick:1986}.  It
is convenient to approach the problem  in a slightly more general setting.

Suppose $Y$ is a regularly varying element of a separable Hilbert
space whose index of regular variation is $p, \ p \in (1,2)$.
In our application, $Y=X\otimes X$,  the Hilbert space is $\cS$
and $p=\ag/2$. Denote by $\mu_Y$ the exponent measure of $Y$
and by $u_N$ a regularly varying sequence such that
$N P(\lnorm Y \rnorm  > u_N) \to 1$, so that
\begin{equation}\label{e:mu_N,Y-def}
\mu_{N, Y}(A) := \frac{P(Y \in u_N A)}{P(\|Y\| >u_N)} \to \mu_Y(A),
\end{equation}
with the usual restrictions on the set $A$, cf. Proposition~\ref{p:RV}.
Set
\[
Y_N = u_N^{-1} N Y I_{\lbr \|Y \| > u_N \rbr}
\]
and observe that $E [ Y_N ] $ exists in the sense of Bochner.  Indeed, by
\refeq{Xtail} and the Potter bounds (Lemma \ref{l:L-properties}), we have
\[
  P(\|Y\|>u) = u^{-p}L(u) = o(u^{-p+\delta}),\ \quad\mbox{ as }u\to\infty,
\]
for an arbitrarily small $\delta>0$.  Since $p\in (1,2)$, by taking $p-\delta>1$, we obtain
$E [\|Y \|] = \int_0^\infty P(\|Y\|>y) dy < \infty$ and the expectation of $Y$ and hence $Y_N$ is well-defined.

Now set $M_N = E [Y_N]$. We want to identify $M \in H$ such that
$\| M_N - M\| \to 0$.  We will show that the above convergence  holds with
\begin{equation}\label{e:M-def}
M = \int_{ \displaystyle {\mathbb B}^c } y \mu_Y(dy),
\end{equation}
where ${\mathbb B}= \lbr y: \| y \| \le 1 \rbr$.  Recall that $Y$ is regularly varying and by \eqref{e:disintegration} its
exponent and angular measures are related as follows
\begin{equation} \label{e:mu-Ga}
\mu_Y(dy) = p r^{-p-1} dr \Gamma_Y(d\theta),
\end{equation}
where $r:= \|y\|$ and $\theta:=y/\|y\|$ are polar coordinates in $H$.
Thus, in polar coordinates, we obtain
\begin{align}\label{e:M}
 \int_{ \displaystyle {\mathbb B}^c } \|y\| \mu_Y(dy)
&= \int_1^\infty \int_{\displaystyle {\mathbb S}}
 r \|\theta \|   \Gamma_Y(d\theta)  p r^{-p -1} dr \\
& = \lp  p \int_1^\infty r^{-p} dr \rp \int_{\displaystyle {\mathbb S}}
\|\theta \| \Gamma_Y(d\theta)\nn \\
&= \frac{p}{p-1}. \nn
\end{align}
This shows that the Bochner integral in \eqref{e:M-def} is well defined and in fact equals
\[
M = \frac{p}{p-1} \int_{\displaystyle
\mathbb{S}} \theta \Gamma_Y(d\theta).
\]
In view of Remark \ref{rem:Gamma_Y}, by taking $Y = X\otimes X$ and
$p=\alpha/2$, we then obtain
\[
M = \frac{\alpha}{\alpha-2} \int_{\displaystyle \mathbb{S}_H}
\lp \theta\otimes \theta \rp \Gamma_X(d\theta),
\]
which is the expression for the offset in \refeq{limC}.

Observe that by the definition \eqref{e:mu_N,Y-def} of $\mu_{N,Y}$, since $N P( \|Y\|>u_N)\to 1$,
for any Bochner integrable mapping of the Hilbert space into itself, or to the real line,
\begin{equation} \label{e:int-f}
N E [ f(u_N^{-1} Y) ] \sim   \int f(y) \mu_{N, Y}(dy).
\end{equation}
Therefore,
\[
M_N = N E \lb u_N^{-1} Y  I_{\displaystyle {\mathbb B}^c}
 (u_N^{-1}Y)\rb \sim \int_{ \displaystyle {\mathbb B}^c } y \mu_{N, Y} (dy).
\]
Observe that $\mu_{N, Y} ( \displaystyle {\mathbb B}^c) = 1$,
and by  \refeq{mu-Ga},
\[
\mu_Y({\mathbb B}^c) =
\int_1^\infty \int_{\displaystyle {\mathbb S}} p r^{-p-1} dr
\Gamma_Y(d\theta) = \sigma_Y(\displaystyle {\mathbb S}) = 1.
\]
Thus $\mu_{N, Y}$ and $\mu_Y$ are probability measures on
$\displaystyle {\mathbb B}^c$, and we want to show that
\[
\int_{ \displaystyle {\mathbb B}^c } y \mu_{N, Y} (dy)
\to \int_{ \displaystyle {\mathbb B}^c } y \mu_{Y} (dy).
\]
Since $\mu_{N, Y}$ converges weakly to $\mu_Y$, it suffices to verify
that
\begin{equation} \label{e:ui}
\sup_{N\ge 1} \int_{ \displaystyle {\mathbb B}^c }
\lnorm y \rnorm^{1+\dg} \mu_{N, Y} (dy) < \infty,
\end{equation}
for some $\dg> 0$ (this implies strong uniform integrability).
Observe that by \refeq{int-f},
\begin{align} \label{e:crystal-ball}
\int_{\displaystyle {\mathbb B}^c}
\lnorm y \rnorm^{1+\dg} \mu_{N, Y} (dy)
&= N E \lb \lnorm u_N^{-1} Y \rnorm^{1+\dg} I_{ \displaystyle {\mathbb B}^c }( u_N^{-1} Y) \rb \nonumber \\
& = N u_N^{-1-\dg} E_N(\dg),
\end{align}
where
\[
 E_N(\dg) = E \lb \lnorm Y \rnorm^{1+\dg}
I_{\{ \lnorm Y\rnorm > u_N\}} \rb.
\]

By the Tonelli--Fubini theorem, we have
\begin{align*}
E_N(\dg) &= E \lp \int_{u_N^{1+\delta}} I_{\{ \|Y\|^{1+\delta} > x\}} dx \rp  = \int_{u_N^{1+\delta}}^\infty P\lp \|Y\|^{1+\delta} > x \rp dx \\
& = \int_{u_N^{1+\delta}}^\infty  x^{-p/(1+\delta)} L(x^{1/(1+\delta)}) dx.
\end{align*}
Now, by picking $\delta>0$ such that $\eta := p/(1+\delta)>1$ and applying the Karamata Theorem (Lemma \ref{l:L-properties}(iii)), for the
right-hand side of \eqref{e:crystal-ball}, we obtain
\begin{align*}
N u_N^{-1-\dg} E_N(\dg) &\sim N u_N^{-1-\dg} \frac{1}{\eta-1} \lp u_N^{1/(1+\delta)} \rp^{1-p/(1+\delta)} L(u_N)\\
& \sim \frac{1}{\eta-1} N u_N^{-p} L(u_N) = \frac{1}{\eta-1} N P(\|Y\|>u_N) \to  \frac{1}{\eta-1},
\end{align*}
where the last convergence follows from the definition of the sequence
$u_N$.  This shows that the supremum in \eqref{e:ui} is finite, which
completes the proof.

\subsection{Proofs of the results of Section~\ref{s:KR}}

\subsubsection*{Proof of Theorem~\ref{t:KR1}}
The results of this section require
Assumptions \ref{a:ass-conv} and \ref{a:la}.

Before stating Theorem~\ref{t:KR1}, we referred to
Lemma~\ref{l:1} which ensures that the the series
\begin{align*}
T_{j, N}
&= \sum_{k\neq j} (\la_j - \la_k)^{-1}
\lip Z_N, v_j \otimes v_k \rip v_k;\\
T_j
&= \sum_{k\neq j} (\la_j - \la_k)^{-1}
\lip Z, v_j \otimes v_k \rip v_k.
\end{align*}
converge a.s. in $L^2$. These series play a fundamental role in our
arguments.

\begin{lemma} \label{l:1} Suppose $\Psi\in \cS$. For $1 \le j \le p$, set
\[
g_j(\Psi) = \sum_{k\neq j} (\la_j - \la_k)^{-1}
\lip \Psi, v_j \otimes v_k \rip v_k.
\]
Then, the series defining $g_j(\Psi)$ converges in $L^2$.
\end{lemma}
\begin{proof}
Since the $v_k$ are orthonormal, it is enough to check that
\[
\sum_{k\neq j} (\la_j - \la_k)^{-2}
\lip \Psi, v_j \otimes v_k \rip^2 < \infty.
\]
Since the system $\lbr v_j\otimes v_k, \ j, k \ge 1\rbr$ forms an
orthonormal basis in $\cS$
\[
\sum_{j, k \ge 1} \lip \Psi, v_j \otimes v_k \rip^2
= \lnorm \Psi \rnorm_\cS^2 < \infty.
\]
Therefore,
\[
\sum_{k\neq j} (\la_j - \la_k)^{-2}
\lip \Psi, v_j \otimes v_k \rip^2 \leq \ag_j^{-2} \lnorm \Psi \rnorm_\cS^2,
\]
withe $\ag_j$ defined in \refeq{ag-j}.
\end{proof}

\rightline{\QED}

We will use the following lemma, which is analogous to Lemma 1 in
\citetext{kokoszka:reimherr:2013}, whose fully analogous
proof, based on algebraic manipulations,  is omitted.

\begin{lemma}\label{l:KR1} For any $j \ge 1$,
\[
\lip \hat v_j - v_j, v_j \rip = -\frac{1}{2} \lnorm \hat v_j - v_j \rnorm^2.
\]
For any $j, k \ge 1$ such that $j \neq k$ and $\hat\la_j \neq \la_k$,
\[
\lip \hat v_j - v_j, v_k \rip = r_N^{-1} (\hat\la_j - \la_k)^{-1}
\lip Z_N, \hat v_j\otimes v_k \rip.
\]
\end{lemma}

By Assumption \ref{a:ass-conv}, $\| \widehat C - C \|_{\cS} =
O_P( r_N^{-1})$. Using the well--known inequalities
\[
|\hat \la_j - \la_j | \leq \| \widehat C - C \|_{\cS}, \ \ \
\lnorm \hat v_j - v_j \rnorm
\le \frac{2\sqrt{2}}{\ag_j} \| \widehat C - C \|_{\cS},
\]
(see e.g. Lemmas 2.2 and 2.3 in \citetext{HKbook}),
we obtain the following Lemma.

\begin{lemma} \label{l:KR2} For $1 \le j \le p$,
\[
\| \widehat C - C \|_{\cS} = O_P( r_N^{-1}), \ \ \
|\hat \la_j - \la_j | = O_P( r_N^{-1}), \ \ \
\lnorm \hat v_j - v_j \rnorm = O_P( r_N^{-1}).
\]
\end{lemma}

\begin{lemma}\label{p:KR1}
For $1 \le j \le p$,
\[
\lnorm r_N(\hat v_j - v_j) - T_{j, N} \rnorm = O_P\lp r_N^{-1}\rp.
\]
\end{lemma}
\begin{proof}
The same arguments apply to any fixed $j \in \lbr 1, 2, \ldots, p \rbr$,
so to reduce the number of indexes used, we present them for $j=1$.
Set
\[
d_{N, k} = \lip r_N(\hat v_1 - v_1) - T_{1, N}, v_k \rip,
\]
where
\[
T_{1, N} = \sum_{\ell\ge 2}  (\la_1 - \la_\ell)^{-1}
\lip Z_N, v_1\otimes v_\ell \rip v_\ell.
\]
By Parseval's identity,
\[
\lnorm r_N(\hat v_j - v_j) - T_{j, N} \rnorm^2
= \sum_{k=1}^\infty d_{N, k}^2.
\]
Focusing on the first term, $k=1$, observe that
\[
\lip T_{1, N}, v_1\rip  = \sum_{\ell\ge 2}  (\la_1 - \la_\ell)^{-1}
\lip Z_N, v_1\otimes v_\ell \rip \lip v_\ell, v_\ell \rip = 0
\]
and, by Lemmas \ref{l:KR1} and \ref{l:KR2},
\[
\lip r_N(\hat v_1 - v_1), v_1\rip
= - \frac{r_N}{2} \lnorm \hat v_1 - v_1\rnorm^2
= O_P(r_N^{-1}).
\]
We conclude that $d_{N,1}^2 O_P(r_N^{-2})$, and it remain to show that
\begin{equation} \label{e:k2}
\sum_{k=2}^\infty d_{N, k}^2 = O_P(r_N^{-2}).
\end{equation}

In the remainder of the proof it is assumed that $k \ge 2$. Since
\[\lip T_{1, N} , v_k\rip
= (\la_1 - \la_k)^{-1} \lip Z_N, v_1\otimes v_k \rip,
\]
by Lemma~\ref{l:KR1},
\[
d_{N, k} = (\hat\la_1 - \la_k)^{-1}
\lip Z_N, \hat v_1 \otimes v_k \rip
- (\la_1 - \la_k)^{-1} \lip Z_N, v_1 \otimes v_k \rip.
\]
Using a  common denominator and rearranging the numerator, we
obtain
\[
d_{N,k} = \frac{ \lip(\la_1 - \la_k) Z_N(\hat v_1 - v_1)
+ (\la_1 - \hat\la_1) Z_N(v_1) \ , \ v_k \rip }
{ (\hat\la_1 - \la_k)^2 (\la_1 - \la_k)^2 }.
\]
It is convenient to decompose the sum in \refeq{k2} as
\[
\sum_{k=2}^\infty d_{N, k}^2 = D_{N,1} + D_{N,2} + D_{N,3},
\]
where
\begin{align*}
D_{N, 1}
&= \sum_{k\ge 2} \frac{ \lip Z_N(\hat v_1 - v_1), v_k\rip^2 }
{(\hat\la_1 - \la_k)^2}, \\
D_{N, 2}
&= \sum_{k\ge 2}
\frac{2  (\la_1 - \hat\la_1) \lip Z_N(\hat v_1 - v_1), v_k\rip
\lip Z_N(v_1),  v_k\rip }
{(\hat\la_1 - \la_k)^2 (\la_1 - \la_k)}, \\
D_{N, 3}
&= \sum_{k\ge 2}
\frac{(\la_1 - \hat\la_1)^2 \lip Z_N(v_1),  v_k\rip^2}
{(\hat\la_1 - \la_k)^2 (\la_1 - \la_k)^2}.
\end{align*}
Since $\hat\la_1 - \la_k \ge \hat\la_1 - \la_2$, by Parseval's identity,
\[
D_{N, 1} \leq \frac{1}{(\hat\la_1 - \la_2)^2} \sum_{k\ge 2}
\lip Z_N(\hat v_1 - v_1), v_k\rip^2
\le \frac{ \| Z_N(\hat v_1 - v_1) \|^2}
{(\hat\la_1 - \la_2)^2}.
\]
By Lemma~\ref{l:KR2}, the denominator converges in
probability to $ (\la_1 - \la_2)^2$, and the
numerator is bounded above by
$\| Z_N\|^2 \| (\hat v_1 - v_1) \|^2 = O_P(r_N^{-2})$.

A similar argument shows that
\[
|D_{N,2}|
\le \left |
\frac{2(\la_1 - \hat\la_1)}{(\hat\la_1 - \la_2)^2 (\la_1 - \la_2) }
\right |
\left | \lip Z_N(\hat v_1 - v_1), Z_N(v_1)\rip \right |.
\]
The denominator again converges to a positive constant.
By the Cauchy--Schwarz inequality,
\[
\left | \lip Z_N(\hat v_1 - v_1), Z_N(v_1)\rip \right |
\le \| Z_N(\hat v_1 - v_1) \| \| Z_N(v_1)\|
\leq \| Z_N\|^2 \|\hat v_1 - v_1\|.
\]
We see that $D_{N,2} = O_P(r_N^{-2})$.

The above  method also shows that $D_{N,3} = O_P(r_N^{-2})$.
\end{proof}

\rightline{\QED}

\medskip

\noindent{\sc Proof of Theorem~\ref{t:KR1}:}\
 To prove the first relation, we use the decomposition
\[
r_N (\hat v_j - v_j )  = T_{j,N}+ \lp r_n (\hat v_j - v_j ) - T_{j, N} \rp.
\]
By Lemma~\ref{p:KR1}, it suffices to show that the
$T_{j, n}$ converge jointly in distribution to the $T_j$.
Consider the operator ${\bf g}: \cS \to (L^2)^p$ defined by
\[
{\bf g}(\Psi) = [ g_1(\Psi), g_2(\Psi), \ldots, g_p(\Psi)]^\top,
\]
with the functions $g_j$ defined in Lemma~\ref{l:1}. The
proof of Lemma~\ref{l:1} shows that
$\| g_j(\Psi) \| \le \ag_j^{-1} \|\Psi \|_\cS$, so each
$g_j$ is a continuous linear operator. Hence ${\bf g}$ is
continuous,  and so ${\bf g}(Z_N) \convd {\bf g}(Z)$.
Since, $g_j(Z_N) = T_{j, N}$ and $g_j(Z) = T_j$, the required
convergence follows.

Now we turn to the convergence of the eigenvalues. We will derive an
analogous decomposition,
\begin{equation} \label{e:dec-bg}
r_N(\hat\la_j - \la_j ) = \lip Z_N(v_j), v_j \rip + \bg_N(j),
\end{equation}
and show that for each $j=1, 2, \ldots, p$, $\bg_N(j) = O_P(r_N^{-1})$.
Since the projections
\[
\cS \ni \Psi \mapsto \lip \Psi(v_j), v_j \rip
= \lip \Psi, v_j \otimes v_j \rip_\cS
\]
are continuous, the claim will follow.

Observe that
\begin{align*}
(\hat\la_j - \la_j) v_j
& = \hat\la_j v_j - \hat\la_j \hat v_j + \hat\la_j \hat v_j - \la_j v_j \\
&= \hat\la_j (v_j - \hat v_j) + \widehat C (\hat v_j) - C(v_j) \\
&= (\widehat C - C)(\hat v_j) + C(\hat v_j - v_j) - \hat\la_j (\hat v_j - v_j).
\end{align*}
It follows that
\[
r_N (\hat\la_j - \la_j) v_j
= Z_N(\hat v_j)  + r_N \lbr C(\hat v_j - v_j) - \hat\la_j (\hat v_j - v_j)\rbr.
\]
We decompose the first term as
$Z_N(\hat v_j) = Z_N(v_j) + Z_N(\hat v_j - v_j)$ and get \refeq{dec-bg}
with
\begin{align*}
\bg_N(j)
&= \lip Z_N(\hat v_j - v_j), v_j \rip
+ r_N \lip C(\hat v_j - v_j) - \hat\la_j (\hat v_j - v_j), v_j \rip\\
&= r_N
\lip \lb (\widehat C - C) + C -\hat\la_j \rb(\hat v_j - v_j), v_j \rip\\
&= r_N
\lip \lb (\widehat C - C) + (C-\la_j) -(\hat\la_j-\la_j) \rb
(\hat v_j - v_j), v_j \rip
\end{align*}
By Lemma~\ref{l:KR2},
\[
\lip (\widehat C - C) (\hat v_j - v_j), v_j \rip = O_P(r_N^{-2})
\]
and
\[
\lip (\hat\la_j - \la_j) (\hat v_j - v_j), v_j \rip = O_P(r_N^{-2}).
\]
Since $C$ is symmetric
\[
\lip (C-\la_j) (\hat v_j - v_j), v_j \rip
= \lip \hat v_j - v_j, (C-\la_j)(v_j) \rip = 0.
\]
This shows that $\bg_N(j) = O_P(r_N^{-1})$, and completes the proof.

\subsubsection*{Proof of Theorem~\ref{t:M}}
We start with a simple lemma,  custom formulated for our needs.

\begin{lemma} \label{l:S1} Suppose $\{ X_n \}$ and $\{ Y_n \}$
are sequences of nonnegative random variables and $\{ a_n \}$ is a
convergent sequence of nonnegative numbers. Suppose
$X_n \le Y_n + a_n$. If the $Y_n$ are uniformly integrable,
then so are the $X_n$.
\end{lemma}
\begin{proof}
We will establish a more general result under the assumption that $C:= \sup_{n\in{\mathbb N}} a_n <\infty$.
 Recall that a sequence $\{X_n\}$ is uniformly integrable if and only if the following two conditions hold

 {\em (i)} We have $\sup_{n\in\mathbb N} E |X_n| <\infty$.

 {\em (ii)} For all $\epsilon>0$, there exists a $\delta>0$, such that
 $$
  \sup_{n\in\mathbb N} E \left( |X_n|1_A\right) <\epsilon,
 $$
 for all events such that $P(A)<\delta$
(see, e.g., Theorem 6.5.1 on page 184 in \citetext{resnick:1999}).

 Since $\{Y_n\}$ is uniformly integrable, we have  $\sup_{n\in \mathbb N} E |Y_n| <\infty$ and Condition (i) above follows from
 the triangle inequality and the boundedness of the sequence $\{a_n\}$.  To show that Condition (ii) holds, observe that by
 the triangle inequality
 \begin{equation}\label{e:Xn_via_Yn_and_C}
 \sup_{n\in\mathbb N}E \left(|X_n| 1_A \right) \le \sup_{n\in\mathbb N}E \left(|Y_n| 1_A \right) + C P(A).
 \end{equation}
 Using the uniform integrability of $\{Y_n\}$, for every $\epsilon>0$, one can find $\delta'>0$ such that
 the first term in the right-hand side of \refeq{Xn_via_Yn_and_C} is less than $\epsilon/2$, provided
 $P(A)<\delta'$.  By setting $\delta:= \min\{\delta',\epsilon/(2C)\}$, we also ensure that the second
 term therein is less than $\epsilon/2$ for all $P(A)<\delta\le \delta'$.  This completes the proof of the uniform integrability of $\{X_n\}$.
\end{proof}

\rightline{\QED}

In the following, we assume that $\ga$ is a fixed number in $(0, \ag/2)$.
Theorem 6.1 of \citetext{acosta:gine:1979} implies that, in the
notation of Theorem~\ref{t:KM}, cf. \refeq{KM-stable},
\[
\lim_{N\to\infty}
E \lnorm b_N^{-1} \lp \sum_{i=1}^N Z_i-\ga_N \rp \rnorm^\gamma
= E\lnorm S \rnorm^\gamma.
\]
Applying the above result to \refeq{asym-prodX}, we obtain
\begin{equation} \label{e:C0}
\lim_{N\to\infty} E \lnorm S_N \rnorm^\gamma
= E\lnorm S \rnorm^\gamma,
\end{equation}
where
\[
S_N  = k_N^{-1} \lp \sum_{i=1}^N X_i \otimes X_i-\psi_N\rp.
\]
In the framework of Theorem~\ref{t:asym-hatC-C}, set
\[
M= \int_{ \displaystyle {\mathbb B}_\cS^c } y \mu_{X\otimes X}(dy)
\]
and
\[
M_N= k_N^{-1} N E \lb (X\otimes X) I_{\lbr \| X \|^2 \ge k_N \rbr} \rb,
\]
so that \refeq{S1} becomes
\[
N k_N^{-1} \lp \widehat C - C \rp  = S_N - M_N
\]
with $S_N \convd S$ and $\| M_N - M \|_\cS \to 0$. We now
explain why we can  conclude that
\begin{equation} \label{e:C1}
E\lnorm S_N - M_N \rnorm_\cS^\ga \to  E\lnorm S - M \rnorm_\cS^\ga.
\end{equation}
Since $S_N - M_N \convd S-M$ in $\cS$, $\lnorm S_N - M_N
\rnorm_\cS^\ga \convd \lnorm S - M \rnorm_\cS^\ga$ in $\mathbb R$.
Convergence \refeq{C1} will follow if we can assert that the
nonnegative random variables $\lnorm S_N - M_N \rnorm_\cS^\ga$ are
uniformly integrable.  Since $\lnorm S_N \rnorm_\cS^\ga \convd \lnorm
S \rnorm_\cS^\ga$ and \refeq{C0} holds, Theorem 3.6 in
\citetext{billingsley:1999} implies that the random variables $\lnorm
S_N \rnorm_\cS^\ga$ are uniformly integrable. Relation~\refeq{C1}
thus follows from the inequality
\[
\lnorm S_N - M_N \rnorm_\cS^\ga \le C_\ga
\lbr \lnorm S_N \rnorm_\cS^\ga  +\lnorm M_N \rnorm_\cS^\ga \rbr
\]
and Lemma~\ref{l:S1}. Relation \refeq{C1} implies the first
relation in Theorem~\ref{t:M} with $L_\ga(N) = L_0^{-2\ga}(N)$.

Since $|\hat\la_j - \la_j | \le \| \widehat C - C \|_\cS$
(see e.g. Lemma 2.2 in \citetext{HKbook}), the second relation
follows from the first. Under Assumption~\ref{a:la},
$\|\hat v_j - v_j \| \le  a_j \| \widehat C - C \|_\cS$ (see e.g.
Lemma 2.3 in \citetext{HKbook} or Lemma 4.3 in \citetext{bosq:2000}),
so the third relation also follows from the first.

\subsection{Proof of Theorem~\ref{t:a}}
Since $\lnorm \Psi_{KL}-\Psi \rnorm_{\mathcal{L}} \to 0$ by
\refeq{L-S} and \refeq{A1}, it is enough to show that
\begin{equation}\label{e:con-YKL}
\lnorm \widehat \Psi_{KL}-\Psi_{KL} \rnorm_{\mathcal{L}}
\convas  0.
\end{equation}
The operators $\Psi_{KL}$  and $\widehat \Psi_{KL}$ have
 the following expansions:
\[
\widehat \Psi_{KL}(x) =  \sum_{k=1}^K \sum_{\ell =1}^L
\frac{\hat\sg_{\ell k}}{\hat\la_\ell}  \lip \hat v_\ell, x \rip \hat u_k,
\ \ \
\Psi_{KL}  (x) =  \sum_{k=1}^K \sum_{\ell =1}^L
\frac{\sg_{\ell k}}{\la_\ell} \lip v_\ell, x \rip u_k.
\]
Introduce the sample analogs of the subspaces $\mathcal V_L$ and
$\mathcal U_K$,
\[
\widehat{\mathcal V}_L = {\rm span}\lbr \hat v_1, \hat v_2, \ldots, \hat v_L \rbr,
\ \ \ \widehat {\mathcal U}_K = {\rm span} \lbr \hat u_1, \hat u_2, \ldots, \hat u_K \rbr,
\]
and consider the following projections:
\[
\pi^L=\text{projection onto} \ {\mathcal V}_L,
\ \ \  \hat \pi^L=\text{projection onto} \ \widehat{\mathcal V}_L;
\]
\[
\pi^K=\text{projection onto} \ {\mathcal U}_K, \ \ \
\hat \pi^K=\text{projection onto} \  \widehat {\mathcal U}_K.
\]
Observe that
\[
\widehat \Psi_{KL} =\hat \pi^K D_N \widehat C^{-1}\hat \pi^L, \\\ \Psi_{KL} =\pi^K D C^{-1} \pi^L,
\]
where
\[
D=E \left[ X \otimes Y \right], \\\ D_N=\frac{1}{N} \sum_{i=1}^N X_i \otimes Y_i,
\]
and
\[
C=\sum_{j=1}^\infty \la_j v_j \otimes v_j, \\\ \widehat C=\sum_{j=1}^\infty \hat \la_j \hat v_j \otimes \hat v_j, \\\ C^{-1}=\sum_{j=1}^\infty \la_j^{-1} v_j \otimes v_j, \\\ \widehat C^{-1}=\sum_{j=1}^\infty \hat \la_j^{-1} \hat v_j \otimes \hat v_j.
\]
Notice that for any $y=\pi^L(x)$ or $y=\hat \pi^L(x)$, $C^{-1}(y)$ and $\widehat C^{-1}(y)$ exist.\\

For $x \in L^2$, consider the decomposition
\[
\begin{split}
\lp \widehat \Psi_{KL}-\Psi_{KL} \rp (x)=& \hat \pi^K D_N \lp \sum_{j=1}^L \hat \la_j^{-1} \lip \hat v_j, x \rip  \hat v_j \rp- \pi^K D \lp \sum_{j=1}^L \la_j^{-1} \lip v_j, x \rip  v_j \rp\\
=& \hat \pi^K D_N \lp \sum_{j=1}^L \lp \hat \la_j^{-1}-\la_j^{-1} \rp \lip \hat v_j, x \rip  \hat v_j \rp\\
& \ \ +\hat \pi^K D_N \lp \sum_{j=1}^L \la_j^{-1} \lip \hat v_j-v_j, x \rip  \hat v_j \rp \\
& \ \ +\hat \pi^K D_N \lp \sum_{j=1}^L \la_j^{-1} \lip v_j, x \rip  \lp \hat v_j-v_j \rp \rp\\
& \ \ + \lp \hat \pi^K D_N -\pi^K D \rp \lp \sum_{j=1}^L \la_j^{-1} \lip v_j, x \rip  v_j \rp\\
=&: a_N(x)+b_N(x)+c_N(x)+d_N(x),
\end{split}
\]
where
\begin{align*}
&
a_N(x)=\hat \pi^K D_N \lp \sum_{j=1}^L \lp \hat \la_j^{-1}-\la_j^{-1} \rp \lip \hat v_j, x \rip  \hat v_j \rp, \\
&
b_N(x)=\hat \pi^K D_N \lp \sum_{j=1}^L \la_j^{-1} \lip \hat v_j-v_j, x \rip  \hat v_j \rp, \\
&
c_N(x)=\hat \pi^K D_N
\lp \sum_{j=1}^L \la_j^{-1} \lip v_j, x \rip  \lp \hat v_j-v_j \rp \rp,\\
&
d_N(x)=\lp \hat \pi^K D_N -\pi^K D \rp
\lp \sum_{j=1}^L \la_j^{-1} \lip v_j, x \rip  v_j \rp.
\end{align*}

Relation \refeq{con-YKL} will follow from Lemmas \ref{l:aN},
\ref{l:bN}, \ref{l:cN} and \ref{l:dN}.
The first two of these lemmas use the following result.

\begin{lemma}\label{l:piK-DN}
Under the assumptions of Theorem~\ref{t:a},
\[
\lnorm \hat \pi^K D_N (\hat v_j) \rnorm \le \hat \la_j^{1/2} \lp \frac{1}{N}\sum_{i=1}^N \lnorm Y_i \rnorm^2  \rp^{1/2}.
\]
\end{lemma}
\begin{proof}
For each integer $\ell$, we have
\[
\begin{split}
|\lip \hat \pi^K D_N (\hat v_j), \hat u_\ell \rip | =& \Big| \lip \sum_{k=1}^K \frac{1}{N} \sum_{i=1}^N \lip X_i,\hat v_j \rip \lip Y_i, \hat u_k \rip \hat u_k, \hat u_\ell \rip \Big| \\
=& \Big| \frac{1}{N} \sum_{i=1}^N \lip X_i,\hat v_j \rip \lip Y_i, \hat u_\ell \rip \Big|\\
\le & \frac{1}{N} \lp \sum_{i=1}^N \lip X_i,\hat v_j \rip^2 \rp^{1/2}\lp \sum_{i=1}^N \lip Y_i,\hat u_\ell \rip^2 \rp^{1/2} \\
=& \lp \lip \widehat C(\hat v_j), \hat v_j \rip \rp^{1/2}
\lp \lip \widehat C_Y(\hat u_\ell), \hat u_\ell \rip \rp^{1/2}\\
=& \hat \la_j^{1/2}\hat\ga_\ell^{1/2}.  \ \ \ \ \
(\hat\ga_\ell = \langle \widehat C_Y(\hat u_\ell), \hat u_\ell \rangle.)
\end{split}
\]
Therefore,
\[
\lnorm \hat \pi^K D_N (\hat v_j) \rnorm=\sum_{\ell=1}^\infty \lip \hat \pi^K D_N (\hat v_j), \hat u_\ell \rip^2 \le \hat \la_j \sum_{\ell=1}^\infty \hat \ga_\ell,
\]
and
\[
\sum_{\ell=1}^\infty \hat \ga_\ell = \sum_{\ell=1}^\infty \lp \frac{1}{N} \sum_{i=1}^N \lip Y_i,\hat u_\ell \rip^2 \rp=\frac{1}{N}\sum_{i=1}^N \lnorm Y_i \rnorm^2.
\]
Hence the claim holds.
\end{proof}

\rightline{\QED}

\begin{lemma}\label{l:aN}
Under the assumptions of Theorem~\ref{t:a},
$\lnorm a_N \rnorm_{\mathcal{L}} \convas  0$.
\end{lemma}
\begin{proof}
Observe that
\[
\begin{split}
\lnorm a_N(x) \rnorm & =\lnorm \hat \pi^K D_N \lp \sum_{j=1}^L \lp \hat \la_j^{-1}-\la_j^{-1} \rp \lip \hat v_j, x \rip  \hat v_j \rp \rnorm\\
& \le \sum_{j=1}^L \frac{| \hat \la_j-\la_j |}{\hat \la_j \la_j} | \lip \hat v_j, x \rip |  \lnorm \hat \pi^K D_N (\hat v_j) \rnorm.
\end{split}
\]
By Lemma~\ref{l:piK-DN}, Lemma 2.2 of \citetext{HKbook} and the Cauchy-Schwarz inequality, we obtain the bound
\[
\begin{split}
\lnorm a_N(x) \rnorm \le & \sum_{j=1}^L \la_j^{-1} \hat \la_j^{-1/2}  | \lip \hat v_j, x \rip |  \lp \frac{1}{N}\sum_{i=1}^N \lnorm Y_i \rnorm^2  \rp^{1/2} \lnorm \widehat C-C \rnorm_{\mathcal{L}} \\
\le & \la_L^{-1} \hat \la_L^{-1/2} \lnorm x \rnorm L^{1/2} \lp \frac{1}{N} \sum_{i=1}^N \lnorm Y_i \rnorm^2  \rp^{1/2} \lnorm \widehat C-C \rnorm_{\mathcal{L}}.
\end{split}
\]
By Corollary~\ref{c:as}, for $N>N_1$ (random),
\[
\hat \la_L \ge \la_L- \lnorm \widehat C-C \rnorm_{\mathcal{L}} \ge \la_L/2.
\]
Then we have
\[
\lnorm a_N \rnorm_{\mathcal{L}} \le \sqrt{2} \lp \frac{1}{N} \sum_{i=1}^N \lnorm Y_i \rnorm^2  \rp^{1/2} \la_L^{-3/2} L^{1/2}  \lnorm \widehat C-C \rnorm_{\mathcal{L}}.
\]
Corollary~\ref{c:as} implies that, for any $\ga \in (1, \alpha/2)$,
$
N^{1-1/\ga} \| \widehat C-C \|_{\cS} \convas 0,
$
and by the strong law of large numbers
\[
\frac{1}{N} \sum_{i=1}^N \lnorm Y_i \rnorm^2 \convas
 E \lnorm Y \rnorm^2 \le 2 \lp \lnorm \Psi \rnorm_\cS^2 E \lnorm X \rnorm^2+E\lnorm \eg \rnorm^2  \rp < \infty.
\]
The  claim thus  follows from condition \refeq{c-a}.
\end{proof}

\rightline{\QED}

\begin{lemma}\label{l:bN}
Under the assumptions of Theorem~\ref{t:a},
$\lnorm b_N \rnorm_{\mathcal{L}} \convas  0$.
\end{lemma}
\begin{proof}
Lemma~\ref{l:piK-DN} implies that
\[
\begin{split}
\lnorm b_N(x)  \rnorm =& \lnorm \hat \pi^K D_N \lp \sum_{j=1}^L \la_j^{-1} \lip \hat v_j-v_j, x \rip  \hat v_j \rp \rnorm \\
\le & \sum_{j=1}^L \la_j^{-1} | \lip \hat v_j-v_j, x \rip |  \lnorm \hat \pi^K D_N (\hat v_j) \rnorm\\
\le & \sum_{j=1}^L \la_j^{-1} \hat \la_j^{1/2} \lnorm x \rnorm \lnorm \hat v_j-v_j \rnorm  \lp \frac{1}{N} \sum_{i=1}^N \lnorm Y_i \rnorm^2  \rp^{1/2}.
\end{split}
\]
Lemma 2.3 of \citetext{HKbook} yields the relation
\[
\lnorm \hat v_j-v_j \rnorm \le 2\sqrt{2} \alpha_j^{-1}
\lnorm \widehat C-C \rnorm_{\mathcal{L}},
\]
with the $\alpha_i$ defined in \refeq{ag-j}. Hence,
\[
\lnorm b_N\rnorm_{\mathcal{L}} \le 2 \sqrt{2} \la_L^{-1} \hat \la_1^{1/2} \lp \sum_{j=1}^L\alpha_j^{-1} \rp  \lp \frac{1}{N} \sum_{i=1}^N \lnorm Y_i \rnorm^2  \rp^{1/2} \lnorm \widehat C-C \rnorm_{\mathcal{L}}.
\]
Since, for $N>N_2$ (random),
\[
\hat \la_1 \le \la_1+ \lnorm \widehat C-C \rnorm_{\mathcal{L}} \le \frac{3}{2} \la_1,
\]
we have
\[
\lnorm b_N \rnorm_{\mathcal{L}} \le 2 \sqrt{3} \la_L^{-1} \la_1^{1/2} \lp \sum_{j=1}^L\alpha_j^{-1} \rp  \lp \frac{1}{N} \sum_{i=1}^N \lnorm Y_i \rnorm^2  \rp^{1/2} \lnorm \widehat C-C \rnorm_{\mathcal{L}}.
\]
By Corollary~\ref{c:as} and the strong law of large numbers,
the claim follows from \refeq{c-b}.
\end{proof}

\rightline{\QED}

\begin{lemma}\label{l:cN}
  Under the assumptions of Theorem~\ref{t:a}, $ \lnorm c_N
  \rnorm_{\mathcal{L}} \convas 0 $.
\end{lemma}
\begin{proof}
Observe that
\[
\begin{split}
\lnorm c_N(x)  \rnorm =& \lnorm \hat \pi^K D_N \lp
\sum_{j=1}^L \la_j^{-1} \lip v_j, x \rip  \lp \hat v_j-v_j \rp \rp \rnorm\\
\le &  \lnorm \hat \pi^K D_N \rnorm_{\mathcal{L}} \sum_{j=1}^L \la_j^{-1} | \lip v_j, x \rip |  \lnorm \hat v_j-v_j \rnorm\\
\le &\lnorm \hat \pi^K D_N \rnorm_{\mathcal{L}}  \lp \sum_{j=1}^L \la_j^{-1} \alpha_j^{-1} | \lip v_j, x \rip | \rp\lnorm \widehat C-C \rnorm_{\mathcal{L}}\\
\le & \lnorm \hat \pi^K D_N \rnorm_{\mathcal{L}}  \la_L^{-1}  \lnorm x \rnorm \lp \sum_{j=1}^L \alpha_j^{-1}\rp\lnorm \widehat C-C \rnorm_{\mathcal{L}}.
\end{split}
\]
Therefore,
\[
\lnorm c_N \rnorm_{\mathcal{L}} \le
\lnorm D_N \rnorm_{\mathcal{L}}  \la_L^{-1}
\lp \sum_{j=1}^L \alpha_j^{-1}\rp
\lnorm \widehat C-C \rnorm_{\mathcal{L}}.
\]
Since, by the law of large numbers,
$
\lnorm \hat \pi^K D_N \rnorm_\cL
 \convas \lnorm D \rnorm_\cL$,
the claim  follows from condition \refeq{c-b}.
\end{proof}

\rightline{\QED}

To deal with the last term, we need additional lemmas.

\begin{lemma} \label{l:D}
Under the assumptions of Theorem~\ref{t:a},
$N^{1-1/\ga} \| D_N - D \|_\cS \convas 0$.
\end{lemma}
\begin{proof}
The decomposition
\[
\frac{1}{N} \sum_{i=1}^N X_i \otimes Y_i=\frac{1}{N} \sum_{i=1}^N X_i \otimes \Psi(X_i)+\frac{1}{N} \sum_{i=1}^N X_i \otimes \eg_i
\]
and the identities
\[
X_i \otimes \Psi(X_i)=\Psi ( X_i \otimes X_i), \ \ E [X \otimes \Psi(X)]
=\Psi E [X \otimes X], \ \ E[X \otimes \eg]=0
\]
imply that
\begin{align*}
 \lnorm D_N-D \rnorm_\cS
&= \lnorm \frac{1}{N}
\sum_{i=1}^N X_i \otimes Y_i-E[X \otimes Y] \rnorm_\cS\\
 & \le
\lnorm \Psi \rnorm_\cS \lnorm \widehat C-C \rnorm_\cS
+\lnorm \frac{1}{N} \sum_{i=1}^N X_i \otimes \eg_i \rnorm_\cS.
 \end{align*}
For any $1 \le \ga < 2$,
\[
\lnorm \frac{1}{N^{1/\ga}} \sum_{i=1}^N  X_i \otimes \eg_i \rnorm_\cS
\convas 0.
\]
The above convergence follows from Theorem 4.1 of
\citetext{acosta:1981} which implies that in any separable Banach
space of Rademacher type $\ga$, $1 \le \ga < 2$, $N^{-1/\ga}
\sum_{i=1}^N Y_i \convas 0$, provided the $Y_i$ are iid with $E \| Y_i
\|^\ga < \infty$ and $ E Y_i = 0$. In our case, the Banach space is
the Hilbert space $\cS$ (a Hilbert space has Rademacher type $\ga$ for
any $\ga \le 2$, see e.g.  Theorems 3.5.2 and 3.5.7 of
\citetext{linde:1986}). Clearly, $E [X_i \otimes \eg_i] = 0$ and $ E
\| X_i \otimes \eg_i \|_\cS^\ga= E \| X_i \|^\ga E\| \eg_i\|^\ga <
\infty $. Another  application of Corollary \ref{c:as} completes the proof.
\end{proof}

\rightline{\QED}

\begin{lemma}\label{l:piK}
Under the assumptions of Theorem~\ref{t:a},
$
\la_L^{-1}
\lnorm \hat \pi^K D_N-\pi^KD \rnorm_{\mathcal{L}}  \convas 0
$.
\end{lemma}
\begin{proof}
By the triangle inequality,
\[
\lnorm \hat \pi^K D_N-\pi^K D \rnorm_\cL
\le \lnorm \hat \pi^K D_N-\hat \pi^K D \rnorm_\cL
+\lnorm \hat \pi^K D-\pi^K D \rnorm_\cL.
\]
For the first term, we have
\begin{align*}
\lnorm \hat \pi^K D_N-\hat \pi^K D \rnorm_\cL
&=  \sup_{\lnorm x\rnorm \le 1} \lnorm \sum_{k=1}^K \lip \lp D_N-D \rp (x) , \hat u_k \rip \hat u_k \rnorm \\
& \le \sup_{\lnorm x\rnorm \le 1}
\lp \sum_{k=1}^K \Big| \lip \lp D_N-D \rp (x), \hat u_k \rip \Big|  \rp \\
& \le K^{1/2} \lnorm D_N-D \rnorm_\cL.
\end{align*}
Thus,
$\la_L^{-1} \lnorm \hat \pi^K D_N-\hat \pi^K D \rnorm_\cL
\convas 0$ by Lemma~\ref{l:D}  and condition \refeq{c-?}.

Turning to the second term, observe first that
\[
D(x) =  E [ \lip X, x \rip Y ]
=\Psi( E [\lip X, x \rip X]) =  \Psi (C(x)).
\]
Setting $y = \Psi (C(x))$ we thus have
\[
\pi^K D(x) = \sum_{k=1}^K \lip y , u_k \rip u_k, \ \ \
\hat \pi^K D(x) = \sum_{k=1}^K \lip y , \hat u_k \rip \hat u_k.
\]
Consequently, $\hat \pi^K D(x) - \pi^K D(x) = D_1(x) + D_2(x)$, where
\begin{align*}
D_1(x) = \sum_{k=1}^K \lip y, u_k - \hat u_k \rip u_k, \ \ \
D_2(x) = \sum_{k=1}^K \lip y, \hat u_k \rip (u_k - \hat u_k).
\end{align*}
Next,
\[
\| D_1(x) \|
\le \| y \| \lbr \sum_{k=1}^K \| u_k - \hat u_k\|^2 \rbr^{1/2}
\le  2 \sqrt{2} \| y \| \lnorm\widehat C_Y - C_Y  \rnorm_\cL
 \lbr \sum_{k=1}^K \frac{1}{\bg_k^2} \rbr^{1/2}
\]
and
\[
\| D_2(x) \|
\le  \sum_{k=1}^K |\lip y , \hat u_k \rip |  \| u_k - \hat u_k\|
\le  2 \sqrt{2} \| y \| \lnorm\widehat C_Y - C_Y  \rnorm_\cL
\sum_{k=1}^K \frac{1}{\bg_k}.
\]
We see that condition \refeq{c-D} implies  that
$\la_L^{-1} \lnorm \hat \pi^K D-\hat \pi^K D \rnorm_\cL
\convas 0$.
\end{proof}

\rightline{\QED}

\begin{lemma}\label{l:dN}  Under the assumptions of Theorem~\ref{t:a},
  $ \lnorm d_N \rnorm_\cL \convas 0$.
\end{lemma}
\begin{proof}
Observe that
\begin{align*}
\lnorm d_N(x)  \rnorm^2
&= \lnorm \lp \hat \pi^K D_N -\pi^K D \rp
\lp \sum_{j=1}^L \la_j^{-1} \lip v_j, x \rip  v_j \rp \rnorm^2\\
 & \le  \lnorm \hat \pi^K D_N-\pi^KD \rnorm_\cL^2
\lp \sum_{j=1}^L \la_j^{-2} \lip v_j, x \rip^2 \rp  \\
& \le  \lnorm \hat \pi^K D_N-\pi^KD \rnorm_\cL^2
\la_L^{-2}\lp \sum_{j=1}^L  \lip v_j, x \rip^2 \rp  \\
&  \le \lnorm \hat \pi^K D_N-\pi^KD \rnorm_\cL^2
\la_L^{-2}  \lnorm x \rnorm^2.
\end{align*}
Consequently,
$
\lnorm d_N \rnorm_\cL \le\lnorm \hat \pi^K D_N-\pi^KD \rnorm_\cL
 \la_L^{-1} ,
$
so the claim follows from Lemma~\ref{l:piK} and condition \refeq{c-a}.
\end{proof}

\rightline{\QED}

\end{document}